\documentclass[a4paper,11pt]{amsart}

\usepackage{amsmath}
\usepackage{amsfonts}
\usepackage{amssymb}
\usepackage{graphicx}
\usepackage[abbrev,alphabetic]{amsrefs}
\usepackage{color}
\usepackage{soul}
\usepackage[rgb,dvipsnames]{xcolor} 
\setstcolor{red}
\usepackage{mathtools}
\usepackage{hyperref}
\hypersetup{
bookmarks,
bookmarksdepth=3,
bookmarksopen,
bookmarksnumbered,
pdfstartview=FitH,
colorlinks,backref,hyperindex,
linkcolor=Sepia,
anchorcolor=BurntOrange,
citecolor=MidnightBlue,
citecolor=OliveGreen,
filecolor=BlueViolet,
menucolor=Yellow,
urlcolor=OliveGreen
}
\usepackage[margin=1.35in]{geometry}

\usepackage{amsthm}
\usepackage{comment}
\usepackage[all,cmtip]{xy}
\usepackage{tikz-cd}
\usetikzlibrary{cd}

\makeatletter
\newcommand*{\relrelbarsep}{.386ex}
\newcommand*{\relrelbar}{%
  \mathrel{%
    \mathpalette\@relrelbar\relrelbarsep
  }%
}
\newcommand*{\@relrelbar}[2]{%
  \raise#2\hbox to 0pt{$\m@th#1\relbar$\hss}%
  \lower#2\hbox{$\m@th#1\relbar$}%
}
\providecommand*{\rightrightarrowsfill@}{%
  \arrowfill@\relrelbar\relrelbar\rightrightarrows
}
\providecommand*{\leftleftarrowsfill@}{%
  \arrowfill@\leftleftarrows\relrelbar\relrelbar
}
\providecommand*{\xrightrightarrows}[2][]{%
  \ext@arrow 0359\rightrightarrowsfill@{#1}{#2}%
}
\providecommand*{\xleftleftarrows}[2][]{%
  \ext@arrow 3095\leftleftarrowsfill@{#1}{#2}%
}
\makeatother

\newcommand{\stacksproj}[1]{{\cite[Tag~{#1}]{stacks-project}}}

\newcommand{\bG}{\mathbb{G}}
\newcommand{\bN}{\mathbb{N}}
\newcommand{\bQ}{\mathbb{Q}}
\newcommand{\bF}{\mathbb{F}}

\newcommand{\bZ}{\mathbb{Z}}
\newcommand{\bP}{\mathbb{P}}
\newcommand{\cA}{\mathcal{A}}
\newcommand{\cB}{\mathcal{B}}
\newcommand{\cC}{\mathcal{C}}
\newcommand{\cD}{\mathcal{D}}

\newcommand{\cF}{\mathcal{F}}

\newcommand{\cH}{\mathcal{H}}

\newcommand{\cI}{\mathcal{I}}

\newcommand{\cL}{\mathcal{L}}

\newcommand{\cM}{\mathcal{M}}
\newcommand{\cO}{\mathcal{O}}

\newcommand{\red}{\mathrm{red}}

\DeclareMathOperator{\cocone}{cocone}
\DeclareMathOperator{\cone}{cone}
\DeclareMathOperator{\hofib}{hofib}
\DeclareMathOperator{\Obj}{Obj}
\DeclareMathOperator{\Mor}{Mor}

\DeclareMathOperator{\Supp}{Supp}
\DeclareMathOperator{\Spec}{Spec}

\DeclareMathOperator{\Hom}{Hom}

\DeclareMathOperator{\Pic}{Pic}

\DeclareMathOperator{\id}{id}

\newcommand*{\coloneq}{\mathrel{\mathop:}=}
\newcommand*{\PicS}{\mathcal{P}\! \mathit{ic}}

\theoremstyle{plain}
\newtheorem{theorem}{Theorem}[section]

\newtheorem{proposition}[theorem]{Proposition}
\newtheorem{lemma}[theorem]{Lemma}
\newtheorem{corollary}[theorem]{Corollary}

\newtheorem{claim}[theorem]{Claim}

\theoremstyle{definition}
\newtheorem{definition}[theorem]{Definition}

\theoremstyle{remark}
\newtheorem{remark}[theorem]{Remark}

\title[Relative semiampleness in mixed characteristic]{Relative semiampleness in mixed characteristic}

\author{Jakub Witaszek} 
\address{Department of Mathematics \\  
University of Michigan\\  
Ann Arbor, MI 48109, USA}
\email{jakubw@umich.edu}

\begin{document}

\begin{abstract}
We show that a nef line bundle on a proper scheme over an excellent base is semiample if and only if it is semiample after restricting to characteristic zero and to positive characteristic. In the process of the proof, we provide a generalisation to mixed characteristic of the fact that the perfection of the Picard functor in positive characteristic is a stack in groupoids for the h-topology.
\end{abstract}


\maketitle

\section{Introduction}

This article is an application of the results and techniques developed in \cite{witaszek2020keels}. We start with an overview from the birational geometric perspective. The readers interested more in the behaviour of line bundles on blow-up squares and the context of h-stacks are referred to Subsection \ref{ss:excision-intro}.

The geometry of an algebraic variety can often be captured by constructing maps to other varieties. In the category of projective varieties such maps correspond to line bundles which are \emph{base point free}, and so it is a fundamental problem to find ways of verifying whether a given line bundle admits such a property. In characteristic zero, this is usually achieved by employing vanishing theorems or analytic methods. In positive characteristic, vanishing theorems are false in general, but their use may often be replaced, sometimes yielding even stronger results, by application of the Frobenius morphism: sending every function to its $p$-th power. This idea underpins the famous Keel's theorem (\cite{keel99}), which in turn allowed Cascini and Tanaka to show that, up to a multiple, the base-point-freeness in positive characteristic can be verified fibrewise.
\begin{theorem}[{\cite[Theorem 1.1]{ct17}}] \label{thm:Cascini-Tanaka}
Let $\pi \colon X \to S$ be a projective morphism of excellent schemes of positive characteristic. Let $L$ be a line bundle on $X$ such that $L|_{X_s}$ is semiample for every point $s \in S$ and the fibre $X_s$ over it. Then $L$ is relatively semiample.
\end{theorem}  
\noindent Here, a line bundle is \emph{semiample} if and only if some multiple of it is base point free. Note that this result is false in characteristic zero. Theorem \ref{thm:Cascini-Tanaka} shows that a line bundle which is trivial on each fibre, descends to the base up to some multiple as long as $\pi$ has geometrically connected fibres; this special case of the theorem was proven originally in \cite{BS17} by Bhatt and Scholze using different methods (see also \cite[Tag 0EXA]{stacks-project}). In fact, they showed that all vector bundles which are trivial on each fibre of $f$ descend to the base after perfection.

Not only does Theorem \ref{thm:Cascini-Tanaka} provide a structurally important uniform description of the behaviour of semiampleness in families, but it is also essential to the future development of the positive characteristic Minimal Model Program. Note that a three-dimensional special case thereof was proven in \cite{bw14} and used to construct Mori fibre spaces, while \cite{HaconWitaszekMMP4fold} employed Theorem \ref{thm:Cascini-Tanaka} to show the validity of the Minimal Model Program for four-dimensional varieties over a curve contingent upon the existence of resolutions of singularities.      

Therefore, it is natural to wonder if an analogue of Theorem \ref{thm:Cascini-Tanaka} holds in mixed characteristic, that is for schemes which are \emph{not} defined over a field. Mixed characteristic schemes bridge the gap between positive and zero characteristics and come naturally in the context of number theory. The study of their geometry and commutative algebra gathered much interest in recent year (cf.\ \cites{andre18,bhatt16,ms18,maschwede18,tanaka16_excellent,EH16,BMPSTWW20,TakamatsuYoshikawaMMP}). Our current project grew as an application  of \cite{witaszek2020keels} in which Keel's theorem on semiampleness and Koll\'ar's theorem on quotients by finite equivalence relations were extended from positive to mixed characteristic (see Theorem \ref{theorem:mixed-char-Keel} and Theorem \ref{thm:quotients}; see also \cite{stigant2021augmented} for some additional results including applications to the augmented base locus).

One of the key goals of this article is to generalise Theorem \ref{thm:Cascini-Tanaka} to mixed characteristic which was mentioned as forthcoming in \cite[Theorem 1.8]{witaszek2020keels}. In particular, we get that a line bundle $L$ on a proper scheme $X$ over a mixed characterististic divisorial valuation ring $R$ is semiample if and only if $L|_{X_{\eta}}$ and $L|_{X_s}$ are semiample where $s$ and $\eta$ and the special point and the generic point of $R$, respectively.
\begin{theorem} \label{thm:main-intro} Let $X$ be a scheme admitting a proper morphism $\pi \colon X \to S$ to an excellent scheme $S$ and let $L$ be a line bundle on $X$. Then $L$ is semiample if and only if $L|_{X_{\bQ}}$ is semiample and $L|_{X_s}$ is semiample for every point $s \in S$ having positive characteristic residue field.
\end{theorem}
\noindent Here, $X_{\bQ} = X \times_{\Spec \bZ} \Spec \bQ$. The assumption that $L|_{X_{\bQ}}$ is semiample is necessary due to the fact that Theorem \ref{thm:Cascini-Tanaka} is false in characteristic zero. In an update to \cite{HaconWitaszekMMP4fold}, this theorem  will be used in the proof of the validity of the four-dimensional semistable Minimal Model Program in mixed characteristic, which will, in turn, provide applications to litability of positive characteristic three-dimensional varieties.\\

The proof of Theorem \ref{thm:Cascini-Tanaka} by Cascini and Tanaka inducts on dimension and consists of three steps: first showing the result for when $L$ is relatively numerically trivial, second for when $X$ is normal, and then finally in full generality. The first step is achieved by a meticulous intricate gluing of semiampleness on partial normalisations, the second step by Keel's theorem, flattening, and the existence of regular alterations, while in the third step the semiample fibration $X \to Z$ is constructed by quotienting the fibration $\tilde X \to \tilde Z$ for the normalisation $\rho \colon \tilde X \to X$ coming from the second step by the finite equivalence relation defining $\rho$ (which is possible by the aforementioned result of Koll\'ar from \cite{kollar12}).

For the second and the third step, we follow the strategy of Cascini and Tanaka, by replacing Keel's theorem on semiampleness of line bundles and Koll\'ar's theorem on the existence of quotients by finite equivalence relations by their mixed characteristic variants obtained in \cite{witaszek2020keels}. The case of Step 1 is more intricate as it is already technically involved in positive characteristic and any approach based on directly replacing the use of Frobenius in the proof of Cascini and Tanaka by the techniques developed in \cite{witaszek2020keels}, albeit likely possible, would add another layer of difficulty.\\   

\subsection{Excision for line bundles}\label{ss:excision-intro}

Our proof of the first step is based instead on the ideas of Bhatt and Scholze \cite{BS17} and that of \cite[Tag 0EXA]{stacks-project} (wherein a more elementary explanation is given). This provides a new insight into the behaviour of line bundles in mixed characteristic, which we believe is interesting by itself. 

Let us observe that in order to study line bundles under geometric or scheme-theoretic operations it is often more beneficial to look at the \emph{Picard groupoid} (in which we keep track of the isomorphisms of line bundles) instead of the \emph{Picard group} (in which isomorphic line bundles are identified).

With that in mind, we can state the result of Bhatt and Scholze in the case of line bundles. They showed that given a projective morphism $g \colon Y \to X$ in characteristic $p>0$ which is an isomorphism ouside of a closed subset $Z \subseteq X$, the following diagram
\begin{equation} \label{diagram:BS-picard}
		\begin{tikzcd} 
			\PicS(X)[1/p] \arrow{r} \arrow{d} & \PicS(Y)[1/p] \arrow{d} \\
			\PicS(Z)[1/p] \arrow{r} & \PicS(E)[1/p]
		\end{tikzcd}
		\end{equation}
is Cartesian, where $\PicS(X)$ denotes the groupoid of line bundles on $X$ and $E = g^{-1}(Z)$. As groupoids constitute a \emph{$2$-category} the notion of Cartesianity here is stronger than that for ordinary categories. Explicitly, the above result says that up to raising line bundles to some big enough power of $p$, a data of a line bundle $L_Y$ on $Y$, a line bundle $L_Z$ on $Z$, and an isomorphism $L_Y|_E \simeq (g|_E)^*L_Z$, gives a unique (up to a unique isomorphism) line bundle $L$ on $X$. This result if \emph{false} for the standard Picard group: the line bundle $L$ depends not only on the line bundles $L_Y$ and $L_Z$ which are isomorphic on $E$, but also on the choice of the isomorphism.

Therefrom, Bhatt and Scholze derived that the association $F \colon X \mapsto \PicS(X)[\frac{1}{p}]$ is a stack in groupoids for the h-topology. This implies that given a proper surjective map $g \colon Y \to X$, a line bundle $L_Y$ on $Y$, and an isomorphism between the pullbacks of $L_Y$ to $Y\times_X Y$ under the two projections such that the cocycle condition holds on $Y \times_X Y \times_X Y$, we get a unique (up to a unique isomorphism) line bundle $L$ on $X$. The notion of h-topology developed by Voyevodsky allows here for a concise treament of the behaviour of line bundles under proper surjective maps. Geometrically, this result, by ways of applying flattening and induction, makes it possible to reduce Theorem \ref{thm:Cascini-Tanaka} in the numerically trivial case to when $f$ is flat, in which case it is easier to descend the line bundle (see \cite[Tag 0EXG]{stacks-project}). \\

We cannot expect the same result to be true in mixed characteristic, as it is false in characteristic zero. The path forward is suggested by the case of thickenings $f \colon Y \to X$ of Noetherian schemes (or, more generally, finite universal homeomorphisms) from \cite{witaszek2020keels}. In characteristic $p>0$, we have that $\PicS(X)[1/p] \simeq \PicS(Y)[1/p]$, because $f$ factors through a power of Frobenius. In mixed characteristic, modulo some small adjustments, \cite[Theorem 1.7]{witaszek2020keels} states that
\begin{equation} \label{diagram:Witaszek-thickening}
\begin{tikzcd}
\PicS(X) \otimes \bQ \arrow{r}{f^*} \arrow{d} & \PicS(Y)\otimes \bQ \arrow{d} \\
\PicS(X_{\bQ}) \otimes \bQ \arrow{r} & \PicS(Y_{\bQ})\otimes \bQ,
\end{tikzcd}
\end{equation}
is Cartesian. This suggests that one could extend the condition in (\ref{diagram:BS-picard}) to a Cartesianity of a 3D diagram obtained by putting together the diagram itself and its copy for the base change of the schemes to $\bQ$. Although technically possible, dealing with such constructs would be quite taxing.

Instead, we pursue the idea to consider the \emph{homotopy fibre} as inspired through the collaboration \cite{aemw}. Pick a line bundle $M$ on $X_{\bQ}$ and set $\underline{\PicS}_M(X) := \hofib(\PicS(X) \to \PicS(X_{\bQ}))$ where the homotopy fibre is taken over $M$. Explicitly, $\underline{\PicS}_M(X)$ is a groupoid with objects being pairs $(L,\phi)$, where $L$ is a line bundle on $X$ and $\phi \colon L|_{X_{\bQ}} \xrightarrow{\simeq} M$ is an isomorphism. One can verify that (\ref{diagram:Witaszek-thickening}) is Cartesian if and only if $\underline{\PicS}_M(X) \otimes \bQ \simeq \underline{\PicS}_{f^*M}(Y) \otimes \bQ$ for every line bundle $M$ on $X_{\bQ}$ (see Lemma \ref{lem:cartiesianity-checked-by-fibres}).

This suggests that (\ref{diagram:BS-picard}) should be Cartesian with $\PicS$ replaced by $\underline{\PicS}$. This is indeed the case, as shown by Proposition \ref{prop:fibre-of-Pic-blow-up-squares}, and in fact $\underline{\PicS}$ is a stack in groupoids for the $h$-topology.


\begin{theorem} \label{thm:fibre-of-Pic-is-h-sheaf}
Let $S$ be a Noetherian scheme and $Sch/S$ be the category of schemes of finite type over $S$. Fix a line bundle $M$ on $S_{\bQ}$. 
Let $\underline{\PicS}_M \colon Sch/S \to Groupoids$ be a pseudofunctor sending $X \in Sch/S$ to the homotopy fibre, over the pullback of $M$ to $X_{\bQ}$, of the restriction morphism 
\[
\PicS(X) \to \PicS(X_{\bQ}).
\]
Then $\underline{\PicS}_M \otimes \bQ$ is a stack in groupoids for the h-topology.
\end{theorem}
\noindent This theorem can be then used to deduce the relatively-numerically-trivial case of Theorem \ref{thm:main-intro}.

An analogous result for K-theory was shown in collaboration \cite{aemw} building on deep results in this field and some ideas of \cite{witaszek2020keels}. Via specialisation, \cite{aemw} should imply that Theorem \ref{thm:fibre-of-Pic-is-h-sheaf} holds for $M = \mathcal{O}_{S_{\bQ}}$.  

Although we treat mostly $2$-categorical phenomena, many ideas, by ways of \cite{BS17}, touch upon the higher categorical way of thinking. We hope that our work will help in promoting these concepts in birational geometry.




\section{Preliminaries}
We refer to \cite{stacks-project} for basic definitions in scheme theory and to \cite[Preliminaries]{witaszek2020keels} for a more thorough treatment of nefness, semiampleness, and EWMness, which is also summarised briefly below. We emphasise that we will \emph{not} deal with set theoretic issues; we implicitly work with sets of an appropriately bouded cardinality.

Recall that schemes of finite type over Noetherian base schemes are Noetherian (\stacksproj{01T6}). If a scheme is excellent, then its normalisation is finite (\stacksproj{0BB5}). Excellent schemes are automatically Noetherian. Given a scheme $X$ we write $X_{\bQ} \coloneq X \times_{\Spec \bZ} \Spec \bQ$ and $X_{\bF_p} \coloneq X \times_{\Spec \bZ} \Spec \bF_p$ where $p>0$ is a prime number. Given a morphism $\pi \colon X \to S$, we denote $\pi|_{X_{\bQ}}$ by $\pi_{\bQ}$.

We say that a morphism of schemes $f \colon X \to Y$ is a \emph{contraction} if it is proper, surjective, and $f_* \cO_X = \cO_Y$. Let $X$ be a proper scheme over a Noetherian base scheme $S$, let $\pi \colon X \to S$ be the projection, and let $L$ be a line bundle on $X$. All the notions below are relative to $S$. We say that $L$ is relatively \emph{nef} if $\mathrm{deg}(L|_C) \geq 0$ for every proper curve $C \subseteq X$ over $S$, 
it is relatively \emph{base point free} if the natural map $\pi^*\pi_*L \to L$ is surjective, it is relatively \emph{semiample} if some multiple of it is base point free, and it is relatively \emph{big} if $L|_{X_{\eta}}$ is big for some generic point $\eta \in f(X)$ and the fibre $X_{\eta}$ over $\eta$.

\begin{lemma}[{\cite[Lemma 2.1]{witaszek2020keels}}] \label{lem:pullback-of-big} Let $f \colon X \to Y$ be a finite surjective map of integral proper schemes over a Noetherian base scheme $S$. Let $L$ be a relatively nef line bundle on $Y$. Then $L$ is relatively big over $S$ if and only if $f^*L$ is relatively big over $S$.
\end{lemma}


We say that $L$ is relatively \emph{EWM} (\emph{endowed with a map}) if there exists a proper $S$-morphism $f \colon X \to Y$ to an algebraic space $Y$ proper over $S$ such that an integral closed subscheme $V \subseteq X$ is contracted (that is, $\dim V > \dim f(V)$) if and only if $L|_V$ is not relatively big. The Stein factorisation of $f$ is unique. The property of $L$ being EWM can be checked affine locally on $S$ (cf.\ \cite[Preliminaries]{witaszek2020keels}).

We remind the reader that flat proper morphisms of Noetherian schemes are equidimensional (\cite[Tag 0D4J]{stacks-project}) and recall a known fact about triangulated categories.
\begin{lemma} \label{lem:octahedral}
Consider the following commutative diagram
\begin{center}
\begin{tikzcd}
A \arrow{r} \arrow{d} & B \arrow{d} \\
C \arrow{r} & D  
\end{tikzcd}
\end{center}
in a triangulated category $\mathcal{T}$. Suppose that the map $\cone(A \to C) \to \cone(B \to D)$ is an isomorphism. Then there exists an exact triangle
\[
A \to B \oplus C \to D \xrightarrow{+1}.
\] 
\end{lemma}
\begin{proof}
By assumptions, we have the following exact triangle
\[
A \to C \to E \xrightarrow{+1},
\] 
where $E := \cone(B \to D)$.
The octahedral axiom applied to
\begin{center}
\begin{tikzcd}[column sep = small, row sep = small]
                            & &                            & & C \arrow{rrdd}            & &                & &  \\
                            & &                            & &                           & &                & &  \\
                            & & B \oplus C \arrow{rruu} \arrow{rrd} & &                           & & E \arrow{rrdd} & & \\
                            & &                            & & D \arrow{rru} \arrow{rrrrd} & &                & &  \\
B \arrow{rruu} \arrow{rrrru}& &                            & &                           & &                & & A[1]. 
\end{tikzcd}
\end{center}
yields the exact triangle
\[
A \to B \oplus C \to D \xrightarrow{+1}.
\]
\end{proof}

\subsection{Groupoids}
A \emph{groupoid} is a category in which every morphism is an isomorphism. The key example considered in this article is the groupoid of line bundles $\PicS(X)$ on a scheme $X$ consisting of line bundles on $X$ together with their isomorphisms.

\emph{Morphisms between groupoids} correspond to functors between the corresponding categories. A \emph{$2$-morphism} (`homotopy') between two such functors is a natural transformation thereof (which in case of groupoids is automatically an isomorphism). This constitues a $2$-category $Groupoids$. In other words, this is a full $2$-subcategory of the $2$-category of categories whose objects are groupoids (equivalently, of the $(2,1)$-category of categories in which all $2$-morphisms are isomorphisms).

For a groupoid $\cA$, we denote by $\pi_0(\cA)$ the set of equivalence classes of the objects of $\cA$ up to isomorphism. Given a fixed point $x \in \cA$, we denote by $\pi_1(\cA,x)$ the set of isomorphisms $\Hom(x,x)$ inside $\cA$. 

A morphism $f \colon \cA \to \cB$ is an \emph{equivalence} if it is an equivalence of categories, and it is a \emph{weak equivalence} if it induces a bijection $f \colon \pi_0(\cA) \xrightarrow{\simeq} \pi_0(\cB)$ and an isomorphism of groups $f \colon \pi_1(\cA,x) \xrightarrow{\simeq} \pi_1(\cB,f(x))$ for every $x \in \cA$.

\begin{theorem}[{\cite[Proposition 4.4]{nlab:groupoid}}] \label{thm:weak-equivalence-of-groupoids} Let $f \colon \cA \to \cB$ be a morphism of groupoids. Then it is an equivalence if and only if it is a weak equivalence.
\end{theorem}


A \emph{commutative diagram} 
\begin{center}
\begin{tikzcd}
\cA \arrow{r}{f} \arrow{d}[swap]{h} & \cB \arrow{d}{g} \arrow[shift left = 0.2em, shorten <= 1em, shorten >= 1em, Rightarrow]{ld}[swap]{F} \\
\cC \arrow{r}[swap]{i} & \cD
\end{tikzcd}
\end{center}
in the $2$-category of groupoids consists of groupoids $\cA$, $\cB$, $\cC$, $\cD$ and morphisms $f$, $g$, $h$, $i$ as above, together with a $2$-morphism $F \colon g \circ f \implies i \circ h$ (cf.\ \cite[Tag 003O]{stacks-project}). We emphasise that $F$ is a part of the data defining a commutative diagram. We will usually drop $F$ from the notation and remember that it is there implicitly.  

\begin{definition}[{\cite[Tag 003Q]{stacks-project}}] We say that a commutative diagram as above is \emph{Cartesian} (also called a \emph{$2$-pullback square}) if for every groupoid $\cM$, morphisms $b \colon \cM \to \cB$ and $c \colon \cM \to \cC$, and a $2$-morphism $G \colon g \circ b \implies i \circ c$, there exists a morphism $a \colon \cM \to \cA$ rendering the diagram below commutative and which is \emph{unique up to a unique homotopy}.  
\begin{center}
\begin{tikzcd}
\cM \arrow[dashed]{rd}{a} \arrow[bend left = 15]{rrd}{b} \arrow[bend right = 15]{rdd}[swap]{c} & & \\
& \cA \arrow{r}{f} \arrow{d}{h} & \cB \arrow{d}{g} \\
& \cC \arrow{r}{i} & \cD
\end{tikzcd}
\end{center}
\end{definition}
\noindent In particular, $a \colon \cM \to \cA$ comes equipped  with $2$-morphisms $B \colon f \circ a \implies b$ and $C \colon h \circ a \implies c$ witnessing the commutativity of the diagrams. In the above definition, being unique up to a unique homotopy means that given $a \colon \cM \to \cA$ and $a' \colon \cM \to \cA$ rendering the above diagram commutative there is a unique $2$-morphism $H \colon a \implies a'$ satisfying: $C' \circ H = C$ and $B' \circ H = B$. Here, $B'$ and $C'$ are the $2$-morphisms witnessing the commutativity for $a'$.

Equivalently, the above commutative diagram is Cartesian, if $\cM \simeq \cC \times_{\cD} \cB$, where $\cC \times_{\cD} \cB = \varprojlim \big( \cC \to \cD \leftarrow \cB \big)$ is a groupoid whose objects are triples $(c,b, \phi)$ where $c \in \cC$, $b \in \cB$ and $\phi \in \Hom(i(c), g(b))$, and morphisms between $(c,b,\phi)$ and $(c',b',\phi')$ consist of isomorphisms $c \simeq c'$ and $b \simeq b'$ which commute with $\phi$ and $\phi'$ in $\cD$ (see \cite[Tag 02X9]{stacks-project}).\\

For a morphism of groupoids $f \colon \cA \to \cB$ we define the \emph{homotopy fibre} $\hofib_b(f)$ over $b \in \cB$ to be the groupoid $\cA \times_{\cB} \{b\}$, where $\{b\}$ denotes a trivial groupoid with a map to $b \in \cB$. Explicitly, the objects of this groupoid are pairs $(a, \phi)$ where $a \in \cA$ and $\phi \in \Hom(f(a),b)$, and morphisms between $(a,\phi)$ and $(a', \phi')$ are given by a morphism $\psi \in \Hom(a,a')$ such that $\phi = \phi' \circ f(\psi)$.

In particular, Theorem \ref{thm:weak-equivalence-of-groupoids} implies that a morphism of groupoids $f \colon \cA \to \cB$ is an equivalence if and only if $\hofib_b(f)$ is equivalent to a trivial groupoid for every $b \in \cB$ (for example, use the long exact sequence of homotopy groups). Here, the essential surjectivity of $f$ (that is, surjectivity on $\pi_0$) is equivalent to $\hofib_b(f)$ being non-empty for every $b \in \cB$.
\begin{lemma} \label{lem:cartiesianity-checked-by-fibres}
A \emph{commutative diagram} of groupoids
\begin{center}
\begin{tikzcd}
\cA \arrow{r}{f} \arrow{d}[swap]{h} & \cB \arrow{d}{g} \\
\cC \arrow{r}[swap]{i} & \cD
\end{tikzcd}
\end{center}
is Cartesian if and only if the induced map $\hofib_{b}(f) \to \hofib_{g(b)}(i)$ is an equivalence for every $b \in \cB$.
\end{lemma}
\begin{proof}
We need to show that the induced map $j \colon \cA \to \cC \times_{\cD} \cB$ is an equivalence if and only if $\hofib_{b}(f) \to \hofib_{g(b)}(i)$ is an equivalence for every $b \in \cB$.
\begin{center}
\begin{tikzcd}
\cA \arrow[dashed]{rd}{j} \arrow[bend left = 15]{rrd}{f} \arrow[bend right = 15]{rdd}[swap]{h} & & \\
& \cC \times_{\cD} \cB \arrow{r}{f'} \arrow{d}{h'} & \cB \arrow{d}{g} \\
& \cC \arrow{r}{i} & \cD
\end{tikzcd}
\end{center}

Suppose that the latter condition holds and pick a point $x \in \cC \times_{\cD} \cB$. We need to show that $\hofib_x(j)$ is a trivial groupoid. Since $\cC \times_{\cD} \cB \times_{\cB} \{f'(x)\} \simeq \cC \times_{\cD} \{f'(x)\}$ (cf.\ \cite[Tag 02XD]{stacks-project}) we can replace $\cA$ and $\cB$ by $A \times_{\cB} \{f'(x)\}$ and $\{f'(x)\}$, respectively, and assume that $\cB = \{b\}$ is a trivial groupoid (we set $b=f'(x)$). Here we also used that base changing the map $j$ via $\cC \times_{\cD} \cB \times_{\cB} \{f'(x)\} \to \cC \times_{\cD} \cB$ does not change the homotopy fibre $\hofib_x(j)$.

By definition, we now have that $\cA \simeq \hofib_b(f)$ and $\cC \times_{\cD} \cB \simeq \hofib_{g(b)}(i)$. Hence, $j$ is an equivalence by assumptions.

The proof in the opposite direction is analogous. 
\end{proof}

Certain groupoids can be studied by means of derived categories. More precisely, strictly symmetric monoidal groupoids (such as the groupoid of line bundles) admit a presentation as simplicial abelian groups with non-trivial $\pi_0$ and $\pi_1$ only, and so via the Dold-Kan correspondence they are equivalent to complexes of cohomological amplitude $[-1,0]$ in the derived category of abelian groups. 

We explicate this construction in an elementary fashion below. Let 
\[
\ldots \to C^{-2} \xrightarrow{d_{2}} C^{-1} \xrightarrow{d_{1}} C^0 \xrightarrow{d_0} C^{1} \to \ldots 
\]
be a complex of abelian groups with possibly non-trivial cohomologies in degrees $0$ and $-1$ only. We can associate to it a groupoid $\cC$ whose objects are given by $\ker(d_0)$ and morphisms by $\Hom(x,x') = \{ y \in C^{-1}/d_{2}(C^{-2}) \, \mid \, d_{1}(y) = x - x'\}$. 

In particular, $\pi_0(\cC) = \cH^0(C^{\bullet})$ and $\pi_1(\cC) = \cH^{-1}(C^{\bullet})$. Moreover, a quasi-isomorphism $C^{\bullet} \simeq C'^{\bullet}$ of complexes induces an equivalence of corresponding groupoids $\cC \simeq \cC'$. 

Note that every complex $C^{\bullet}$ as above is quasi-isomorphic to a complex $C^{-1} \xrightarrow{d} C^0$ of length two and its corresponding groupoid has objects given by $C^0$ and morphisms between $x, x' \in C^0$ given by $y \in C^{-1}$ such that $d(y) = x-x'$.

\begin{remark} \label{remark:groupoids-and-cocones} We can also express the homotopy fibre over the trivial element in terms of derived categories. Let $f \colon C^{\bullet} \to C'^{\bullet}$ be a map of two complexes of abelian groups, each with possibly non-trivial cohomologies in degrees $0$ and $-1$ only. Then $\tau_{\leq 0}\mathrm{cocone}(f) \simeq \hofib_0(f)$ as groupoids, where $\mathrm{cocone}(f)=\mathrm{cone}(f)[-1]$ and in $\hofib_0(f)$ we identify $f$ with the map of associated groupoids as in the above construction.
\end{remark}

The following lemma generalises the fact that the Picard group $\Pic(X)$ is isomorphic to $H^1(X, \cO_X^*)$.

\begin{lemma} \label{lem:picard-groupoid-in-terms-of-derived-category} Let $X$ be a quasi-compact quasi-separated scheme. Then
\[
\PicS(X) \simeq \tau_{\leq 0}(R\Gamma(X, \cO_X^*)[1]),
\] 
where the right hand side is a groupoid via the above construction.
\end{lemma}
\begin{proof}
We have that $R\Gamma(X, \cO_X^*)[1]$ is quasi-isomorphic to
\[
\varinjlim_{\{U_i\}} \bigoplus_i H^0(U_i, \cO^*_{U_i}) \to \varinjlim_{\{U_i\}} \bigoplus_{i,j} H^0(U_i \cap U_j, \cO^*_{U_i \cap U_j}) \to \ldots, 
\]
with the first term in degree $-1$, where the limit is taken over all refinements of affine covers (cf.\ \cite[Tag 09UY]{stacks-project}). Indeed, the natural map from this complex to $R\Gamma(X, \cO_X^*)$ (cf.\ \cite[Tag 01FD]{stacks-project}) induces an isomorphism on cohomologies (\cite[Tag 09V2]{stacks-project}), and so is a quasi-isomorphism. Here we used that cohomology commutes with filtered colimits for quasi-compact quasi-separated schemes (\cite[Tag 073E]{stacks-project}). 

In particular, the associated groupoid $\tau_{\leq 0}(R\Gamma(X, \cO_X^*)[1])$ has objects consisting of an affine covering $\{U_i\}$ (up to refinement) and transition functions $\{\phi_{i,j} \in H^0(U_i \cap U_j, \cO^*_{U_i \cap U_j}) \}$ (up to refinement) such that $\phi_{i,j} \cdot \phi_{j,k} = \phi_{i,k}$,  and morphisms given by $\{\psi_{i} \in H^0(U_i, \cO^*_{U_i}) \}$ (up to refinement) sending $\phi_{i,j} \mapsto \psi_i \cdot \phi_{i,j} \cdot \psi^{-1}_j$. This groupoid is equivalent to the groupoid of line bundles. 
\end{proof} 

Last, we discuss stacks in groupoids. To this end, we need the following definition from \cite[Tag 003N]{stacks-project}.
\begin{definition} Let $\cA$ be a category and let $\cB$ be a $2$-category. A pseudo-functor $\cF \colon \cA \to \cB$ is a collection of the following data:
\begin{itemize}
	\item a map $\cF \colon \Obj(\cA) \to \Obj(\cB)$,
	\item for every $x, y \in \Obj(\cA)$ and $f \colon x \to y$, a morphism $\cF(f) \colon \cF(x) \to \cF(y)$,
	\item for every $x,y,z \in \Obj(\cA)$ and morphisms $f \colon x \to y$ and $g \colon y \to z$, a $2$-morphism $G_{g,f} \colon \cF(g \circ f) \implies \cF(g) \circ \cF(f)$ which is an isomorphism,  
	\item for every $x \in \Obj(\cA)$, a $2$-morphism $G_x \colon \id_{\cF(x)} \implies \cF(\id)$,

\end{itemize}
which satisfy the compatibility conditions as in \cite[Tag 003N]{stacks-project}. A contravariant pseudofunctor is a pseudofunctor $\cF \colon \cA^{opp} \to \cB$. 
\end{definition}
The main difference between a functor and a pseudo-functor is that $\cF(g \circ f)$ is not equal to $\cF(g) \circ \cF(f)$ but is only isomorphic to it (via $G_{g,f}$). The association $\PicS \colon X \mapsto \PicS(X)$ is an example of a contravariant pseudofunctor 
\[
\PicS \colon Sch \to Groupoids,
\]
where $Sch$ is the category of schemes. Here, given maps $f \colon X \to Y$ and $g \colon Y \to Z$, there is a canonical $2$-morphism $G_{g,f} \colon \PicS(g \circ f) \implies \PicS(f) \circ \PicS(g)$ induced by canonical isomorphisms $f^*(g^*L) \simeq (g \circ f)^*L$ for every $L \in \PicS(Z)$. 

The datum of a pseudo-functor is equivalent to that of a category fibred in groupoids together with a choice of a cleavage. In particular, one can see that the following definition of a stack in groupoids is equivalent to the one from \cite[Tag 02ZH]{stacks-project}. 

\begin{definition} A stack in groupoids for the Zariski topology is a contravariant pseudo-functor $\cF \colon Sch \to Groupoids$ such that for every scheme $X$ and an affine covering $\{U_i\}$ of $X$, we have that
\[
\cF(X) \simeq \varprojlim \big(\cF(U) \rightrightarrows \cF(U \times_X U) \mathrel{\substack{\textstyle\rightarrow\\[-0.6ex]
                      \textstyle\rightarrow \\[-0.6ex]
                      \textstyle\rightarrow}} \cF(U \times_X U \times_X U)\big),
\]
where $U = \bigsqcup_i U_i$.
\end{definition}
This is equivalent to saying that every descent datum is effective and $\Hom(x,y)$ is a sheaf for $x, y \in \cF(X)$. 

Here, $C := \varprojlim \big(\cF(U) \rightrightarrows \cF(U \times_X U) \mathrel{\substack{\textstyle\rightarrow\\[-0.6ex]\textstyle\rightarrow \\[-0.6ex] \textstyle\rightarrow}} \cF(U \times_X U \times_X U)\big)$ is a groupoid whose objects are pairs $(u,\phi)$, where $u \in \cF(U)$ and $\phi \colon \pi_1^*(u) \xrightarrow{\simeq} \pi_2^*(u)$ is an isomorphism with $\pi_1,\pi_2 \colon U \times_X U \rightrightarrows U$ being the projections, such that $\phi$ satisfies the cocycle condition  in $\cF(U \times_X U \times_X U)$. The morphisms between $(u,\phi)$ and $(u',\phi')$ in $C$ are given by morphisms between $u$ and $u'$ which are compatible with $\phi$ and $\phi'$.

Informally speaking, $\cF(X) \to C$ is an isomorphism exactly when for every $c \in C$ there exists a unique up to a unique isomorphism $x \in \cF(X)$ with the pullback in $C$ isomorphic to $c$. 

\subsection{Sheaves and stacks for the h-topology}
Given a category $\cC$, a \emph{family of morphisms with fixed target in $\cC$} consists of an object $X \in \cC$, a set $I$, and for each $i \in I$, a morphism $X_i \to X$ in $\cC$. We denote it by $\{X_i \to X\}_{i \in I}$. A \emph{site} is a category $\cC$ together with a set $\mathrm{Cov}(\cC)$ of families of morphisms with fixed target satisfying the properties as in \cite[Tag 00VH]{stacks-project}.

Let $Sch/S$ be the category of schemes of finite type over a Noetherian base scheme $S$. We say that a finite family of morphisms with fixed target $\{X_i \to X\}_{i\in I}$ in $Sch/S$ is an \emph{h-covering} if $\bigsqcup_{i \in I} X_i \to X$ is universally submersive (cf.\ \cite[Tag 040H]{stacks-project}). This defines the \emph{h-site} $(Sch/S)_h$. Note that fppf coverings are h-coverings (\cite[Tag 0ETV]{stacks-project}), and so are proper surjective maps in $Sch/S$ (\cite[Tag 0ETW]{stacks-project}). 

Since we restricted ourselves to schemes of finite type over a Noetherian scheme, we only need to deal with finite families of morphisms. In general, the definition of an $h$-covering may be found in \cite[Tag 0ETS and 0ETT]{stacks-project}.    



\begin{definition}Let $t \in \{\mathrm{Zariski}, \mathrm{\acute{e}tale}, \mathrm{fppf}, \mathrm{h}\}$. A contravariant functor $\cF \colon Sch/S \to Sets$ is a \emph{sheaf for the t-topology} if and only if for every finite t-covering $\{U_i \to X\}_{i \in I}$ we have that $\cF(X)$ is the equaliser in 
\[
\cF(X) \to \cF(U) \rightrightarrows \cF(U \times_X U), 
\]
where $U = \bigsqcup_{i \in I} U_i$.

A contravariant pseudo-functor $\cF \colon Sch/S \to Groupoids$ is a \emph{stack in groupoids for the t-topology} if and only if for every finite t-covering $\{U_i \to X\}_{i \in I}$, we have that 
\[
\cF(X) \simeq \varprojlim \big(\cF(U) \rightrightarrows \cF(U \times_X U) \mathrel{\substack{\textstyle\rightarrow\\[-0.6ex]
                      \textstyle\rightarrow \\[-0.6ex]
                      \textstyle\rightarrow}} \cF(U \times_X U \times_X U)\big).
\]

\end{definition}
An important property of sheaves (resp.\ stacks in groupoids) $\cF$ for the t-topology is that given a t-covering $f \colon Y \to X$ we have that $f^* \colon \cF(X) \to \cF(Y)$ is injective (resp.\ faithful). 

For simplicity, given a sheaf (resp.\ stack in groupoids) $\cF$ for the t-topology, a morphism $f \colon Y \to X$, and an object $x \in \cF(X)$, we will sometimes denote $f^*x \in \cF(Y)$ by $x|_Y$.


\begin{remark}
Note that if $\cF$ is a sheaf (resp.\ stack in groupoids) for the h-topology, then $\cF(X) \to \cF(X_{red})$ induced by the reduction map $X_{\red} \to X$ is a bijection (resp.\ an equivalence). Indeed, $X_{\red} \times_X X_{\red} \simeq X_{\red}$ and $X_{\red} \times_X X_{\red} \times_X X_{\red} \simeq X_{\red}$. Moreover, $\cF(X) \to \cF(Y)$ is a bijection (resp.\ an equivalence) for every universal homeomorphism $Y \to X$ as $(Y \times_X Y)_{\red} \simeq Y_{\red}$ and $(Y \times_X Y \times_X Y)_{\red} \simeq Y_{\red}$. 
\end{remark}
The following theorem allows for verifying that a functor (resp.\ pseudo-functor) is a sheaf (resp.\ stack in groupoids) for the h-topology. 
\begin{theorem}[{\cite[Theorem 2.9]{BS17}, cf.\ \cite[Proposition 2.8]{BS17}}] \label{thm:BS-criterion-for-h-sheaves} Let $\cF$ be a functor from $Sch/S$ to $Sets$ (resp.\ pseudo-functor from $Sch/S$ to $Groupoids$). Then $\cF$ is a sheaf (resp.\ stack in groupoids) for the h-topology if and only if the following conditions are satisfied:
\begin{enumerate}
	\item $\cF$ is a sheaf (resp.\ stack in groupoids) for the fppf topology, and
	\item  the following diagram is Cartesian:
		\begin{center}
		\begin{tikzcd}
			\cF(X) \arrow{r} \arrow{d} & \cF(Y) \arrow{d} \\
			\cF(Z) \arrow{r} & \cF(E),
		\end{tikzcd}
		\end{center}
		for every affine scheme $X \in Sch/S$ and a proper surjective map $Y \to X$ which is an isomorphism outside of a closed subset $Z \subseteq X$ with preimage $E \subseteq Y$.
\end{enumerate}
\end{theorem}

We explain some properties of stacks in groupoids for the h-topology that will be used later. The first two pertain to stacks which satisfy an additional (strong) condition that pullbacks under certain maps are fully faithful. 
\begin{lemma} \label{lem:essential-image-h-topology} Let $f \colon Y \to X$ and $g \colon X' \to X$ be proper surjective morphisms of Noetherian schemes. Suppose that $f$ has geometrically connected fibres. Let 
\begin{center}
\begin{tikzcd}
Y' \arrow{r}{g_Y} \arrow{d}{f'} & Y \arrow{d}{f} \\
X' \arrow{r}{g} & X,  
\end{tikzcd}
\end{center}
be a Cartesian diagram.  Let $\cF$ be a stack in groupoids for the h-topology on $Sch/X$ such that for every proper surjective map with geometrically connected fibres $h \colon W \to Z$  of finite type schemes over $X$, the morphism $h^* \colon \cF(Z) \to \cF(W)$ is fully faithful.

Then $\xi \in \cF(Y)$ is in the essential image of $f^* \colon \cF(X) \to \cF(Y)$ if and only if $g_Y^*\xi$ is in the essential image of $f'^* \colon \cF(X') \to \cF(Y')$.
\end{lemma}
\noindent The reader is encouraged to first consider an analogous statement for $h$-sheaves with full faithfulness replaced by injectivity.
\begin{proof}
Set $\xi' = g_Y^*\xi$. The implication from the left to the right is automatic, so we focus on the other one. Let $\eta' \in \cF(X')$ be such that $f'^*\eta' \simeq \xi'$ and consider the following diagram
\begin{center}
\begin{tikzcd}
\cF(Y' \times_Y Y' \times_Y Y')  & \arrow[shift left = 2]{l}  \arrow{l} \arrow[shift right = 2]{l} \cF(Y' \times_Y Y')&   \arrow[shift left = 1]{l}  \arrow[shift right = 1]{l}   \cF(Y')   & \arrow{l}{g^*_Y} \cF(Y)  \\
\cF(X'\times_X X' \times_X X') \arrow{u}  & \arrow[shift left = 2]{l} \arrow{l} \arrow[shift right = 2]{l} \cF(X'\times_X X') \arrow{u}  & \arrow[shift left = 1]{l} \arrow[shift right = 1]{l}  \cF(X')  \arrow{u}{f'^*} &  \arrow{l}{g^*} \cF(X) \arrow{u}{f^*}. 
\end{tikzcd}
\end{center}
Note that the vertical arrows are fully faithful as they are induced by proper surjective maps with geometrically connected fibres (here we use that $(X'\times_X X') \times_X Y \simeq Y'\times_Y Y'$ and analogously for the triple product).
Since $\xi$ induces an object in $\varprojlim \big(\cF(Y') \rightrightarrows \cF(Y'\times_Y Y') \mathrel{\substack{\textstyle\rightarrow\\[-0.6ex] \textstyle\rightarrow \\[-0.6ex]\textstyle\rightarrow}} \cF(Y'\times_Y Y' \times_Y Y') \big)$ extending $\xi'$ and the vertical arrows are fully faithful, we get an induced object in
\[
\varprojlim \big(\cF(X') \rightrightarrows \cF(X'\times_X X') \mathrel{\substack{\textstyle\rightarrow\\[-0.6ex] \textstyle\rightarrow \\[-0.6ex]\textstyle\rightarrow}} \cF(X'\times_X X' \times_X X') \big)
\]  
extending $\eta'$. Explicitly, the isomorphism between the pullbacks of $\xi'$ to $\cF(Y'\times_Y Y')$ descends uniquely to an isomorphism between pullbacks of $\eta'$ in $\cF(X'\times_X X')$, and since the former satisfies the cocycle condition in $\cF(Y'\times_Y Y' \times_Y Y')$, so does the latter in $\cF(X'\times_X X' \times_X X')$. 

The above limit is equivalent to $\cF(X)$ by the defining property of stacks for the h-topology, and so we get $\eta \in \cF(X)$ such that $g^*\eta \simeq \eta'$. The pullback of $\eta$ in  $\varprojlim \big(\cF(Y') \rightrightarrows \cF(Y'\times_Y Y') \mathrel{\substack{\textstyle\rightarrow\\[-0.6ex] \textstyle\rightarrow \\[-0.6ex]\textstyle\rightarrow}} \cF(Y'\times_Y Y' \times_Y Y') \big) \simeq \cF(Y)$ is isomorphic to  $\xi$ by construction.
\end{proof}


\begin{lemma} \label{lem:essential-image-h-topology2}
 Let $f \colon Y \to X$ be a proper surjective morphism of Noetherian schemes and let $\cF$ be a stack in groupoids for the h-topology on $Sch/X$. 
Let $Y= Y_1 \cup Y_2$ for closed subschemes $Y_1$ and $Y_2$. Suppose that $f(Y_1\cap Y_2) = X_1 \cap X_2$ and $(f|_{Y_1\cap Y_2})^* \colon \cF(X_1 \cap X_2) \to \cF(Y_1 \cap Y_2)$ is fully faithful, where $X_1 = f(Y_1)$ and $X_2 = f(Y_2)$. 

 Then $\xi \in \cF(Y)$ is in the essential image of $\cF(X) \to \cF(Y)$ if and only if $\xi|_{Y_1}$ is in the essential image of $\cF(X_1) \to \cF(Y_1)$ and $\xi|_{Y_2}$ is in the essential image of $\cF(X_2) \to \cF(Y_2)$.
 \end{lemma}


 \begin{proof}
 The implication from the left to the right is automatic, and so we show the other one. 
Write $f_1 = f|_{Y_1}$,$f_2 = f|_{Y_2}$, $\xi_1 = \xi|_{Y_1}$, and $\xi_2 = \xi|_{Y_2}$. Let $\eta_1 \in \cF(X_1)$ and $\eta_2 \in \cF(X_2)$ be such that $f_1^*\eta_1 \simeq \xi_1$ and $f_2^*\eta_2 \simeq \xi_2$. Let us call these isomorphisms $\phi_1$ and $\phi_2$, respectively. Also, denote the canonical isomorphism $\xi_1|_{Y_1\cap Y_2} \simeq \xi_2|_{Y_1 \cap Y_2}$ by $\gamma$. 

Since $f^* \colon \cF(X_1 \cap X_2) \to \cF(Y_1\cap Y_2)$ is fully faithful, the isomorphism $f_1^*\eta_1|_{Y_1\cap Y_2} \simeq f_2^*\eta_2|_{Y_1\cap Y_2}$ equal to $\phi^{-1}_2 \circ \gamma \circ \phi_1$ descends to an isomorphism
 \[
\eta_1|_{X_1 \cap X_2} \simeq \eta_2|_{X_1 \cap X_2}.
 \] 
 Since $X_1 \sqcup X_2 \to X$ is an h-cover, the objects $\eta_1$, $\eta_2$ and this isomorphism induce an object $\eta \in \cF(X)$. Moreover, $f^*\eta \simeq \xi$ as $\xi$ is determined up to an isomorphism by its restrictions to $Y_1$ and $Y_2$ and the isomorphism thereof on $Y_1\cap Y_2$. 

 \end{proof}

 Last, we make the following observation. We refer to the introduction for the motivation for the construction of $\underline{\cF}$.
 \begin{lemma} \label{lem:homotopy-fibre-of-t-stack-is-t-stack} Let $\cF$ be a stack in groupoids for the t-topology on $Sch/S$ where $S$ is a Noetherian scheme and $t \in \{ \mathrm{Zariski}, \mathrm{\acute{e}tale}, \mathrm{fppf}, \mathrm{h}\}$. Fix an element $m \in \cF(S_{\bQ})$ and define a pseudofunctor $\underline{\cF} \colon Sch/S \to Groupoids$ sending $X \mapsto \hofib(\cF(X) \to \cF(X_{\bQ}))$ where the homotopy fibre is taken over $m|_{X_{\bQ}} \in \cF(X_{\bQ})$. Then $\underline{\cF}$ is a stack in groupoids for the t-topology.

 The same holds for sheaves $\cF$ for the t-topology, where $\underline{\cF}(X)$ is the inverse image of $m|_{X_{\bQ}}$ under $\cF(X) \to \cF(X_{\bQ})$.
 \end{lemma}
 \begin{proof}
 We focus on the case of stacks as the case of sheaves is analogous (and much simpler).
Let $f \colon Y \to X$ be a covering for the t-topology and consider the following diagram
\begin{center}
\begin{tikzcd}
\cF(Y_{\bQ}\times_{X_{\bQ}} Y_{\bQ} \times_{X_{\bQ}} Y_{\bQ})  & \arrow[shift left = 2]{l} \arrow{l} \arrow[shift right = 2]{l} \cF(Y_{\bQ}\times_{X_{\bQ}} Y_{\bQ}) & \arrow[shift left = 1]{l} \arrow[shift right = 1]{l}  \cF(Y_{\bQ})  &  \arrow{l} \cF(X_{\bQ}) \\ 
\cF(Y \times_X Y \times_X Y) \arrow{u} & \arrow[shift left = 2]{l}  \arrow{l} \arrow[shift right = 2]{l} \cF(Y \times_X Y) \arrow{u} &   \arrow[shift left = 1]{l}  \arrow[shift right = 1]{l}   \cF(Y) \arrow{u}  & \arrow{l} \cF(X) \arrow{u} \\
\underline{\cF}(Y \times_X Y \times_X Y) \arrow{u}  & \arrow[shift left = 2]{l}  \arrow{l} \arrow[shift right = 2]{l} \underline{\cF}(Y \times_X Y) \arrow{u} &   \arrow[shift left = 1]{l}  \arrow[shift right = 1]{l}   \underline{\cF}(Y) \arrow{u}   & \arrow{u} \arrow{l} \underline{\cF}(X). 
\end{tikzcd}
\end{center}
Take $\underline{\eta} \in \varprojlim \big(\underline{\cF}(Y) \rightrightarrows \underline{\cF}(Y\times_X Y) \mathrel{\substack{\textstyle\rightarrow\\[-0.6ex] \textstyle\rightarrow \\[-0.6ex]\textstyle\rightarrow}}\underline{\cF}(Y\times_X Y \times_X Y) \big)$. It induces $\eta_Y \in \varprojlim \big({\cF}(Y) \rightrightarrows {\cF}(Y\times_X Y) \mathrel{\substack{\textstyle\rightarrow\\[-0.6ex] \textstyle\rightarrow \\[-0.6ex]\textstyle\rightarrow}}{\cF}(Y\times_X Y \times_X Y) \big)$ which by the definition of a stack is a pullback of $\eta \in \cF(X)$. 

Moreover, by definition of homotopy fibres, $\underline{\eta}$ also yields an object 
\[
\sigma_{Y_{\bQ}} \in \varprojlim \big(\Hom(\eta|_{Y_{\bQ}}, m|_{Y_{\bQ}}) \rightrightarrows \Hom(\eta|_{Y_{\bQ}\times_{X_{\bQ}} Y_{\bQ}}, m|_{Y_{\bQ}\times_{X_{\bQ}} Y_{\bQ}})\big),
\]
which by the property of stacks for the upper row in the first diagram descends to $\sigma \in \Hom(\eta|_{X_{\bQ}}, m|_{X_{\bQ}})$. Finally, $\eta$ and $\sigma$ yield an object of $\underline{\cF}(X)$. We leave the verification that this object in unique up to a unique isomorphism to the reader.
 \end{proof}

\subsection{Recollection of some results from Cascini-Tanaka}
In this subsection, we recall some general results proven in \cite{ct17}.
\begin{lemma}[{\cite[Lemma 2.11]{ct17}}] \label{lem:semiample-under-pullback}
Let $h \colon X' \to X$ be a proper morphism of proper schemes over a Noetherian base scheme $S$. Let $L$ be a line bundle on $X$. If $L$ is semiample over $S$, then so is $h^*L$. Moreover, the converse is true if $h_*\cO_{X'} = \cO_X$ or $S$ is excellent, $X$ is normal, and $h$ is surjective.
\end{lemma}
\begin{lemma}[{\cite[Lemma 2.12]{ct17}}] \label{lem:semiample-under-faithfully-flat-cover} Let $X$ be a proper scheme over a Noetherian base scheme $S$ and let $L$ be a line bundle on $X$. Let $g \colon S' \to S$ be a morphism of Noetherian schemes. If $L$ is semiample over $S$, then its base change to $X \times_{S} S'$ is semiample over $S'$. Moreover, the converse is true if $g$ is faithfully flat.
\end{lemma} 


The following is a consequence of Gabber's alteration theorem for quasi-excellent schemes. 
\begin{theorem}[{\cite[Theorem 2.30]{ct17}}] \label{thm:regular-alteration} Let $X$ be a normal quasi-excellent scheme. Then there exist a sequence of morphisms $X_m \to X_{m-1} \to \ldots \to X_0$ such that $X_m$ is regular, $X_0=X$, and each map is \'etale surjective, or proper, surjective, and generically finite. 
\end{theorem}

The following result allows for descending line bundles when the fibres are equidimensional, the source is normal integral, and the base is $\bQ$-factorial (in fact, in applications we will render the base regular by ways of Theorem \ref{thm:regular-alteration}). Compare Lemma \ref{lem:descend-equidimensional-base-qfactorial} with the conjunction of \cite[Tag 0EXF and 0EXG]{stacks-project} which assuming flatness, additionally require the line bundle to be torsion on fibres over all codimension one points, but do not require the source to be normal integral nor the base to be $\bQ$-factorial.
\begin{lemma}[{\cite[Lemma 2.17]{ct17}}] \label{lem:descend-equidimensional-base-qfactorial} Let $X$ be an integral normal scheme admitting a proper morphism $\pi \colon X \to S$ to an integral normal excellent scheme $S$ such that $\pi_*\cO_X = \cO_S$. Suppose that $\pi$ is equidimensional and $S$ is $\bQ$-factorial. Let $L$ be a relatively nef line bundle on $X$ such that $L|_{X_{\eta}} \sim_{\bQ} 0$ where $X_{\eta}$ is the fibre over the generic point $\eta \in S$. 

Then $L$ is relatively semiample. In fact, there exists a line bundle $M$ on $S$ and $m \in \bN$ such that $L^m \simeq \pi^*M$.
\end{lemma}

Last, in the case of contractions, descend up to a multiple for a relatively numerically trivial line bundle is equivalent to relative semiampleness.
\begin{lemma}[{cf.\ \cite[Lemma 2.16]{ct17}}] \label{lem:descend-relatively-torsion} Let $\pi \colon X \to S$ be a proper morphism of Noetherian schemes such that $\pi_* \cO_X = \cO_S$. Let $L$ be a relatively numerically trivial line bundle on $X$. Then $L$ is relatively semiample if and only if it is relatively torsion if and only if $L^m \simeq f^*M$ for some line bundle $M$ on $S$ and a natural number $m$.
\end{lemma}
\noindent We say that $L$ is \emph{relatively torsion} if and only if $L^m$ is relatively trivial for some $m \in \bN$, that is, for every $s \in S$, there exists an open neigbhourhood $s \in U \subseteq S$ such that $L^m|_{\pi^{-1}(U)}$ is trivial.
\begin{proof}
The only part which is not automatic is to show that if $L^m$ is relatively trivial for some $m\in \bN$, then it descends to $S$. In this case $\pi_*L^m$ is a line bundle as $\pi_*\cO_X = \cO_S$. We have that $\pi^*\pi_*L^m \to L^m$ is an isomorphism.
\end{proof}


\subsection{Mixed characteristic Keel's theorem}
For a relatively nef line bundle $L$ on a proper scheme $X$ over a Noetherian base scheme $S$, we define $\mathbb{E}(L)$ to be the union of all closed integral subschemes $V \subseteq X$ such that $L|_V$ is not relatively big. 
\begin{theorem}[{\cite[Theorem 1.1]{witaszek2020keels}}] \label{theorem:mixed-char-Keel}  Let $L$ be a nef line bundle on a scheme $X$ projective over an excellent base scheme $S$. Then $L$ is semiample over $S$ if and only if both $L|_{\mathbb{E}(L)}$ and $L|_{X_{\bQ}}$ are so.
\end{theorem} 

The following result is vital in this article. It can be proven analogously to \cite[Theorem 1.10]{witaszek2020keels}; we present a slightly different variant thereof. The reader is referred to \cite[Subsection 2.4]{witaszek2020keels} for the definition and properties of topological and geometric pushouts by universal homeomorphisms. In what follows, we consider a category of pairs $(X,L_X)$ consisting of a scheme with a line bundle $L_X$ on it, and we denote by $f \colon (X,L_X) \to (Y,L_Y)$ a data of a morphism $f \colon X \to Y$ together with an isomorphism $f^*L_Y \simeq L_X$. 



\begin{theorem} \label{thm:universal-homeomorphism-semiample} Let $L$ be a nef line bundle on a scheme $X$ proper over a Noetherian base scheme $S$. Let $f \colon Y \to X$ be a finite universal homeomorphism. Then $L$ is semiample (or EWM) if and only if both $f^*L$ and $L|_{X_{\bQ}}$ are so.
\end{theorem}
\begin{proof}
Provided that $f^*L$ and $L|_{X_{\bQ}}$ are semiample (or EWM), we show that $L$ is semiample (or EWM). We may assume that $S$ is affine. We start with the EWM case of the theorem. Let $g \colon Y \to Z$ be a map associated to $f^*L$. We claim that there exists a topological pushout $Z'$ of $X \leftarrow Y \to Z$ which is proper over $S$:
\begin{center}
\begin{tikzcd}
X \arrow{d}{h} & \arrow{d}{g} \arrow{l}{f} Y\\
Z' & \arrow{l}{r} Z.
\end{tikzcd}
\end{center}
To this end, it is enough to show that $X_{\bQ} \leftarrow Y_{\bQ} \to Z_{\bQ}$ admits a topological pushout, as then the claim will follow from \cite[Theorem 4.4 and Lemma 2.22]{witaszek2020keels}. Let $h_{\bQ} \colon X_{\bQ} \to Z'_{\bQ}$ be a contraction associated to $L|_{X_{\bQ}}$. The induced map $Z_{\bQ} \to Z'_{\bQ}$ is proper (\cite[Tag 04NX]{stacks-project}) and a bijection on geometric points, hence a finite universal homeomorphism (\cite[Tag 0A4X]{stacks-project}). Thus, $Z'_{\bQ}$ is a topological pushout of $X_{\bQ} \leftarrow Y_{\bQ} \to Z_{\bQ}$. Now, the induced map $h \colon X \to Z'$ is one associated to $L$.

We move to the semiample case. We may assume that $h$ constructed above is a contraction. We claim that up to replacing $L$ by a multiple, the above diagram can be extended to a Cartesian diagram of pairs:
\begin{center}
\begin{tikzcd}
(X,L) \arrow{d}{h} & \arrow{d}{g} \arrow{l}{f} (Y,f^*L)\\
(Z',A') & \arrow{l}{r} (Z,A),
\end{tikzcd}
\end{center}
where $A$ is an ample line bundle induced by $g \colon Y \to Z$; explicitly, there exists an isomorphism $\sigma_Y \colon g^*A \xrightarrow{\simeq} f^*L$. This claim immediately concludes the proof of the theorem.  Indeed, since $r$ is finite (it is a finite universal homeomorphism) and $r^*A' \simeq A$, we get that $A'$ is ample (\cite[Tag 0GFB]{stacks-project}) and hence semiample (\cite[Tag 01VS]{stacks-project}). Since $h^*A' \simeq L$, we get that $L$ is semiample as well.

To show the claim, we use the fact that the following diagram 
\begin{center}
\begin{tikzcd}
\PicS(Z')  \arrow{d}  \arrow{r}  &  \PicS(X) \times_{\PicS_Y}  \PicS(Z)  \arrow{d} \\
\PicS(Z'_{\bQ})  \arrow{r} & \PicS(X_\bQ) \times_{\PicS(Y_{\bQ})} \PicS(Z_{\bQ}), 
\end{tikzcd}
\end{center}
is Cartesian up to tensoring by $\bQ$ (as $\PicS$ is a stack in the \'etale topology, we may assume that $Z'$ is a scheme, and so this is \cite[Corollary 3.7]{witaszek2020keels} except for the fact that $S$ therein is assumed to be defined over $\mathbb{Z}_{(p)}$; this is irrelevant in the proof though, for example in view of Theorem \ref{thm:Pic-under-uh}). Let $A'_{\bQ}$ be an ample line bundle on $Z'_{\bQ}$ induced by the contraction $h_{\bQ} \colon X_{\bQ} \to Z'_{\bQ}$. Explicitly, there exists an isomorphism $\sigma_{X_{\bQ}} \colon (h_{\bQ})^* A'_{\bQ} \xrightarrow{\simeq} L|_{X_{\bQ}}$. Further, denote the canonical isomorphism $r^*_{\bQ} A'_{\bQ} \xrightarrow{\simeq} A_{\bQ}$ by $\sigma_{Z,\bQ}$. Then, we get the following element
\[
(A'_{\bQ}, (L, A, \sigma_Y), (\sigma_{X_{\bQ}}, \sigma_{Z,\bQ})) 
\]   
of $\PicS(Z'_{\bQ}) \times_{\PicS(X_\bQ) \times_{\PicS(Y_{\bQ})} \PicS(Z_{\bQ})} \PicS(X) \times_{\PicS(Y)}  \PicS(Z)$, and so up to replacing the line bundles by their multiples, the claim is proven.
\end{proof}

\begin{corollary} \label{cor:connected-fibres-semiample} Let $L$ be a nef line bundle on a scheme $X$ proper over a Noetherian base scheme $S$. Let $f \colon Y \to X$ be a proper surjective morphism with geometrically connected fibres. Then $L$ is semiample if and only if both $f^*L$ and $L|_{X_{\bQ}}$ are so.
\end{corollary}
\begin{proof}
By Stein factorisation, this follows from Lemma \ref{lem:semiample-under-pullback} and Theorem \ref{thm:universal-homeomorphism-semiample}.
\end{proof}

\subsection{Quotients by finite equivalence relations}
\label{sub:quotients_by_finite_equivalence_relations}

We recall definitions and results on set theoretic finite equivalence relations following \cites{witaszek2020keels,kollar12}.

\begin{definition} Let $X$ be a separated algebraic space of finite type over a Noetherian base scheme $S$. A morphism $\sigma \colon E \to X \times_S X$ (equivalently $\sigma_1, \sigma_2 \colon E \rightrightarrows X$ over $S$) is a \emph{set theoretic equivalence relation} on $X$ over $S$ if for every geometric point $\Spec K \to S$ the map
\[
\sigma(K) \colon \Mor_S(\Spec K, E) \to \Mor_S(\Spec K, X) \times \Mor_S(\Spec K, X) 
\] 
yields an equivalence relation on $K$-points of $X$. We say that $\sigma \colon E \to X \times_S X$  is a \emph{finite equivalence relation} if $\sigma_i$ are finite.
\end{definition}
We refer to \cite[Definition 2]{kollar12} for another equivalent definition.

\begin{definition} Let $\sigma_1, \sigma_2 \colon E \rightrightarrows X$ be a set theoretic finite equivalence relation of separated algebraic spaces of finite type over a Noetherian base scheme $S$. We call $q \colon X \to Y$, for a separated algebraic space $Y$ of finite type over $S$, a \emph{categorical quotient}  if $q \circ \sigma_1 = q \circ \sigma_2$ and $q$ is universal with this property (in the category of separated algebraic spaces of finite type over $S$). We call $q$ a \emph{geometric quotient} if
\begin{itemize}
	\item it is a categorical quotient, 
	\item it is finite, and 
	\item for every geometric point $\Spec K \to S$, the fibres of $q_K \colon X_K(K) \to Y_K(K)$ are the $\sigma(E_K(K))$-equivalence classes of $X_K(K)$.
\end{itemize}
\end{definition}
The following result guarantees that quotients exist when $X$ is integral over the base.
\begin{proposition}[{\cite[Lemma 17]{kollar12}, \cite[Proposition 2.13]{witaszek2020keels}}] \label{prop:kollar_quotients_exist} Let $X$ be a separated algebraic space of finite type over a Noetherian base scheme $S$, let $Y$ be an algebraic space over $S$, let $\pi \colon X \to Y$ be an integral morphism over $S$, and let $E \rightrightarrows X$ be a finite, set theoretic, equivalence relation over $Y$. Then the geometric quotient $X/E$ exists as a separated algebraic space of finite type over $S$.
\end{proposition}

It was shown by Koll\'ar (\cite{kollar12}) that quotients by finite set theoretic equivalences always exist in positive characteristic. This was further generalised to mixed characteristic in \cite{witaszek2020keels}.

\begin{theorem}[{\cite[Theorem 1.4]{witaszek2020keels}}] \label{thm:quotients} Let $X$ be a separated algebraic space of finite type over an excellent base scheme $S$. Let $\sigma \colon E \rightrightarrows X$ be a finite, set theoretic equivalence relation and assume that the geometric quotient $X_{\bQ} / E_{\bQ}$ exists as a separated algebraic space of finite type over $S$. Then the geometric quotient $X / E$ exists as a separated algebraic space of finite type over $S$.
\end{theorem}

Last we discuss closures of relations. First, consider a subset $E \subseteq V \times V$, where $V$ is a set. We say that $\overline{E} \subseteq V \times V$ is the \emph{closure} of $E$ if it is the smallest equivalence relation containing $E$. 

\begin{lemma} \label{lem:closure-of-equivalence-relations-for-sets} Consider a subset $E \subseteq V \times V$, where $V$ is a set. Then the closure of $E$ is equal to $\bigcup_{i\geq 0} E_i$, where 
\begin{enumerate}
	\item $E_0 = E \cup \rho(E) \cup \Delta$, for $\Delta \subseteq V \times V$ being the diagonal and $\rho \colon V \times V \to V \times V$ being the involution sending $(v_1,v_2) \in V \times V$ to $(v_2,v_1)$,
	\item $E_i$ is constructed inductively, by setting $E_i = E_{i-1} \times_{V} E_{i-1} \subseteq V\times V$ where the maps $E_{i-1} \rightrightarrows V$ are given by projecting onto the second and the first factor, respectively.
\end{enumerate}
Moreover, if there exists a natural number $d\geq 2$ such that $|\pi_1^{-1}(v)| \leq d$ for every $v \in V$, where $\pi_1 \colon \overline{E} \to V$ is the projection onto the first factor, then $\overline{E} = E_{d-2}$.
\end{lemma} 
\begin{proof}
The first part (the construction of $\overline{E}$) is clear by definition of an equivalence relation: (1) guarantees reflexivity and symmetry, and (2) guarantees transitivity.

To show that $\overline{E} = E_{d-2}$, we need to argue that given distinct $v, v' \in \overline{E}$, the shortest sequence of elements $v_0, \ldots, v_k$ such that $v_0=v$, $v_k=v'$, and $(v_{i-1},v_{i}) \in E_0$ for $1 \leq i \leq k$, satisfies $k-1 \leq d-2$. This is clear, because $v_0, \ldots, v_k$ must be distinct, and so $(v_0,v_0), (v_0,v_1),\ldots, (v_0,v_k) \in \pi_1^{-1}(v_0) \in \overline{E}$.
\end{proof}


Now, let $X$ be a separated scheme of finite type over a Noetherian base scheme $S$ and let $E \subseteq X \times_S X$ be a closed subscheme. We say that a closed subscheme $\overline{E} \subseteq X \times_S X$ is the \emph{closure} of $E$ as a set theoretic equivalence relation, if it is the smallest closed subscheme containing $E$ and defining a set theoretic equivalence relation on $X$.

The following lemma shows that if the closure of $E$ as a finite set theoretic equivalence relation exists over each piece of some stratification of $S$, then it exists in general. 
\begin{lemma} \label{lem:closures-of-subsets-of-finite-equivalence-relations} Let $X$ be a reduced proper scheme over a Noetherian base scheme $S$ (we denote the projection by $\pi \colon X \to S$) and let $E \subseteq X \times_S X$ be a reduced closed subscheme. Suppose that there exists a sequence of closed subschemes $\emptyset =: S_0 \subseteq S_1 \subseteq \ldots \subseteq S_m := S$ such that the reduction of $E_{U_j} := E \times_S U_j \subseteq V_j \times_{U_j} V_j$ is a closed subscheme of a finite set theoretic equivalence relation $H_j \subseteq V_j \times_{U_j} V_j$ for every $1 \leq j \leq m$, where $U_j= S_j \,\backslash\,S_{j-1}$ and $V_j = \pi^{-1}(U_j)\subseteq X$.

Then the closure $\overline{E} \subseteq X \times_S X$ of $E$ exists (as a set theoretic equivalence relation) and is a finite set theoretic equivalence relation.
\end{lemma}
\noindent The notation above makes sense  as (locally closed) immersions are stable under base change (\cite[Tag 01JY]{stacks-project}), hence $V_j = \pi^{-1}(U_j) \simeq X \times_S U_j$ and $V_j \times_{U_j} V_j \simeq (X \times_S X) \times_S U_j$.
\begin{proof}
As above, set $E_0 = E \cup \rho(E) \cup \Delta$, where $\Delta \subseteq X \times_S X$ is the diagonal and $\rho \colon X \times_S X \to X \times_S X$ is the standard involution. We construct $E_i$ inductively, by setting $E_i$ to be the reduction of the image of $E_{i-1} \times_{X} E_{i-1}$ in $X\times_S X$ where the maps $E_{i-1} \rightrightarrows X$ are given by projecting onto the second and the first factor, respectively. Then $E_0 \subseteq E_1 \subseteq E_2 \subseteq \ldots$ and we claim that this sequence stabilises at some point $r \in \bN$. Then $\overline{E} = \bigcup_{i\geq 0} E_i = E_r$ is a set theoretic equivalence relation, which is in fact finite as will be shown in the proof below.

For the claim, pick $d_j$ to be the maximum among the degrees of the fibres of $H_j \to V_j$ (this number exists by \cite[Tag 03JA]{stacks-project}). We will show that $E_i = E_{i+1}$ for all $i\geq r-2$, where $r$ is the maximum among $d_j$ for $1 \leq j \leq m$. To this end, it is enough to show that $E_i(K) = E_{i+1}(K)$ for $i \geq r-2$ and  every geometric point $\Spec K \to S$ with $K$ algebraically closed (here, by abuse of notation, $E_i(K)$ denotes the $K$-points of $E_i$ over the fixed $K$-point of $S$). We may assume that the image of $\Spec K \to S$ is contained in $U_j$ for some $1 \leq j \leq m$. 

Then $E(K) \subseteq V_j(K) \times V_j(K)$ and $\bigcup_{i\geq 0} E_i(K)$ is the closure $\overline{E(K)}$ of $E(K)$ (here, $E_i(K)$ can be constructed directly from $E(K)$ as in the statement of Lemma \ref{lem:closure-of-equivalence-relations-for-sets}, that is $E_i(K) = (E(K))_i$). Since $H_j(K)$ is an equivalence relation on $V_j(K)$ and $E(K)\subseteq H_j(K)$, we have that $\overline{E(K)} \subseteq H_j(K)$. In particular, the fibres of the projection $\overline{E(K)} \to V_j(K)$ onto the first factor are contained in the fibres of $H_j(K) \to V_j(K)$, and so their size is bounded by $d_j$. The claim now follows by Lemma \ref{lem:closure-of-equivalence-relations-for-sets}. \qedhere


\end{proof}


\subsection{Conductors and blow-ups} \label{ss:conductors}

Consider a commutative diagram
\begin{center}
\begin{tikzcd}
Y \arrow{d}{f} & E \arrow{d}{g} \arrow{l}{i} \\
X  & Z \arrow{l}{j}
\end{tikzcd}
\end{center}
where $f$ is a finite map of reduced Noetherian schemes, $Z \subseteq X$ is a closed subscheme, $E \subseteq Y$ is the scheme theoretic inverse image of $Z$, and $f$ is an isomorphism over the complement of $Z$.

For every finite surjective map $f$ of reduced Noetherian schemes, we can construct such a diagram with $Z$ being the \emph{conductor}. Affine locally, we define it as follows. Let $R \subseteq S$ be a finite extension of rings and set $I = \{ s \in S \,\mid\, sS \subseteq R\}$. The ideals $I\subseteq R$ and $I = IS \subseteq S$ define the conductors $C$ and $D$ (then set $Z := C$ and $E := D$). 

Note that $R \simeq S \times_{S/I} R/I$; this is an example of a Milnor square. In fact, that given any ideal $J \subseteq I$, we have that $\phi \colon R \simeq S \times_{S/J} R/J$, which, in turn, yields
\[
R^* \simeq S^* \times_{(S/J)^*} (R/J)^*.
\]   
Indeed, if $r \in R$ is invertible, then its image in $S \times_{S/J} R/J$ is invertible, too, and so we only need to show that given $(s, [r]) \in S^* \times_{(S/J)^*} (R/J)^*$, the corresponding element $r \in R$ via the isomorphism $\phi$ is invertible. But the inverse of an element is unique, so if $s' \in S$ is the inverse of $s$ and $[r'] \in R/J$ is the inverse of $[r]$, then $(s', [r'])$ is an element of $S \times_{S/J} R/J$, and so is the inverse of $(s,[r])$. The corresponding element $r' \in R$ is the inverse of $r$.

In particular, the above paragraph shows the following.
\begin{lemma} \label{lemma:Milnor_Gm}
For a finite surjective map $f$ of reduced Noetherian schemes and a diagram as above, suppose that $Z$ contains the conductor scheme scheme-theoretically. Then the induced diagram
\begin{center}
		\begin{tikzcd}
			\cO^*_X\arrow{r} \arrow{d} & \cO_Y^*\arrow{d} \\
			\cO^*_Z \arrow{r} & \cO^*_E
		\end{tikzcd}
\end{center}
is Cartesian.
\end{lemma}
\noindent In fact, a more general result holds.

\begin{lemma} \label{lemma:Milnor_Gm2}
Let $f \colon Y \to X$ be a finite map of reduced Noetherian schemes and let $Z \subseteq X$ a closed subscheme defined by an ideal sheaf $\cI$. Suppose that $f$ is an isomorphism over the complement of $Z$. Then the induced diagram
\begin{center}
    \begin{tikzcd}
      \cO^*_X\arrow{r} \arrow{d} & \cO_{Y}^*\arrow{d} \\
      \cO^*_{Z_n} \arrow{r} & \cO^*_{E_n}
    \end{tikzcd}
\end{center}
is Cartesian for some $n\gg 0$, where $Z_n \subseteq X$ is a subscheme defined by $\cI^n$ and $E_n$ is its scheme theoretic inverse image.
\end{lemma}
\begin{proof}
Let $X' = f(Y)$ and $Z'_n = Z_n \cap X'$. Consider the following diagram
\begin{center}
    \begin{tikzcd}
      \cO^*_X\arrow{r} \arrow{d} & \cO^*_{X'}\arrow{r} \arrow{d} & \cO_{Y}^*\arrow{d} \\
      \cO^*_{Z_n} \arrow{r} & \cO^*_{Z'_n} \arrow{r} & \cO^*_{E_n}.
    \end{tikzcd}
\end{center}
Since $X = X' \cup Z_n$ and $Z'_n = X' \cap Z_n$, we get that $\cO_X = \ker( \cO_{Z_n} \oplus \cO_{X'} \to \cO_{Z'_n})$ (as $X$ is reduced; this is analogous to \cite[Tag 0CTJ]{stacks-project}), and so the left diagram is Cartesian by same argument as that above Lemma \ref{lemma:Milnor_Gm}. The right diagram is Cartesian by Lemma \ref{lemma:Milnor_Gm} as $Y \to X'$ is surjective and $Z'_n$ contains the conductor for $n\gg 0$. 
\end{proof}
Analogues of these results also hold for $\PicS$ (cf.\ \cite[Lemma 2.28]{witaszek2020keels}). In what follows we will need a variant thereof for blow-up squares.  The proof is based on the idea of \cite[Lemma 4.6]{BS17}.

\begin{lemma} \label{lemma:blow-up-square-Pic}
Let $f \colon Y \to X$ be a projective map of Noetherian schemes which is a blow-up of $X$ with respect to an ideal sheaf $\cI$. Suppose that $X=\Spec A$ is affine and $\cI$-adically complete. Then the induced map
\[
\PicS(X) \to \PicS(Y) \times_{\PicS(E_n)} \PicS(Z_n)
\]
is essentially surjective for $n\gg 0$, where $Z_n \subseteq X$ is the subscheme defined by $\cI^n$ and $E_n$ is its scheme theoretic inverse image.
\end{lemma}
\begin{proof}
We need to show that for $n \gg 0$ every $(L_Y, L_{Z_n}, \phi) \in \PicS(Y) \times_{\PicS(E_n)} \PicS(Z_n)$, where $\phi \colon L_Y|_{E_n} \xrightarrow{\simeq} (f|_{E_n})^*L_{Z_n}$,  is isomorphic to the image of some $L \in \PicS(X)$. 

By affineness of $Z_n$ and deformation theory, there exists a line bundle $L$ on $X$ such that $L|_{Z_n} \simeq L_{Z_n}$. By composing the pullback of this isomorphism with $\phi$, we get an isomorphism $\psi_n \colon L_Y|_{E_n} \xrightarrow{\simeq} f^*L|_{E_n}$. 

To conclude the proof it is enough to show that $\psi_n$ extends to an isomorphism $\psi \colon L_Y \xrightarrow{\simeq} f^*L$. By Grothendieck's existence theorem (\cite[Tag 0885]{stacks-project}), this would follow if we found compatible lifts $\psi_m \colon L_Y|_{E_m} \xrightarrow{\simeq} f^*L|_{E_m}$ for every $m\geq n$, where $E_m$ is the scheme theoretic inverse image of the subscheme $Z_m \subseteq X$ defined by $\cI^m$. To lift $\psi_m$ to $\psi_{m+1}$ (which must then be an isomorphism, as being an isomorphism depends only on the reduced structure), it suffices to prove that
\[
H^0(E_{m+1}, {\cH}om(L_Y|_{E_{m+1}}, f^*L|_{E_{m+1}})) \to H^0(E_{m}, {\cH}om(L_Y|_{E_{m}}, f^*L|_{E_{m}})) 
\]
is surjective for every $m\geq n \gg 0$, and so it is enough to show that
\begin{align*}
H^1(E_{m+1}, \cI^m\! \otimes\! {\cH}om(L_Y|_{E_{m+1}}, f^*L|_{E_{m+1}}))\! &=\! H^1(E_1, \cI^m\!/\cI^{m+1}\! \otimes\! {\cH}om(L_Y|_{E_1}, f^*L|_{E_1})) \\
&\simeq\! H^1(E_1, \cI^m\!/\cI^{m+1} \!\otimes\! {\cH}om(f^*L|_{Z_1}, f^*L|_{Z_1}))
\end{align*}
is zero for $m\geq n \gg 0$. Here ${\cH}om(f^*L|_{Z_1}, f^*L|_{Z_1})$ is locally (over $Z_1$) isomorphic to $\cO_{E_1}$, and so, by affineness of $Z_1$, it is enough to show that $R^1(f|_{E_1})_* \cI^m/\cI^{m+1} = 0$, which by affineness of $Z_1$ again reduces to showing that $H^1(E_1, \cI^m/\cI^{m+1})=0$. This is true by Serre vanishing as $\cI^m/\cI^{m+1} = \cO_{E_1}(m)$ and $\cO_{E_1}(1)$ is relatively ample. We emphasise that the bound from Serre's vanishing does not depend on the choice of $(L_Y, L_{Z_n}, \phi)$.
\end{proof}

\section{Picard groupoid in the $h$-topology}
The goal of this section is to prove Theorem \ref{thm:fibre-of-Pic-is-h-sheaf}.
Given a fixed object $M \in \PicS(S_{\bQ})$, we denote by $\underline{\PicS}_M \colon Sch/S \to Groupoids$ the pseudo-functor such that 
\[
\underline{\PicS}_M(X) = \hofib_{M|_{X_{\bQ}}}(\PicS(X) \to \PicS(X_{\bQ}))
\]
for $X \in Sch/S$. Given $s \in \bG_m(S_{\bQ})$, we denote by $\underline{\bG}_{m,s} \colon Sch/S \to Sets$ a functor such that
\[
\underline{\bG}_{m,s}(X) = \{ t \in \bG_m(X) \, \mid\, t|_{X_{\bQ}} = s|_{X_{\bQ}}\}.
\]
We write $\underline{\PicS}(X) := \underline{\PicS}_{\cO_{S_{\bQ}}}(X)$ and $\underline{\bG}_m := \underline{\bG}_{m,1}$.

The functor $\underline{\bG}_{m,s}$ is a sheaf for the fppf topology and the pseudo-functor $\underline{\PicS}_M$ is a stack in groupoids for the fppf topology (cf.\ Lemma \ref{lem:homotopy-fibre-of-t-stack-is-t-stack}), as so are $\bG_m$ and $\PicS$, respectively (\cite[Tag 04WN]{stacks-project}). In fact, $\underline{\PicS}$ is a stack in (strictly) symmetric monoidal groupoids. That is not the case, however, for $\underline{\PicS}_M$ and a fixed $M \in \PicS(S_{\bQ})$ which is non-trivial. 
\begin{lemma} \label{lemma:kernel-of-Picard-groupoids-in-derived-categories} We have that $\underline{\bG}_m(X) = \mathrm{ker}(H^0(X, \cO_X^*)\to H^0(X_{\bQ}, \cO^*_{X_{\bQ}}))$ and
\[
\underline{\PicS}(X) \simeq \tau_{\leq 0}\cocone(R\Gamma(X, \cO_X^*)[1]\to R\Gamma(X_{\bQ}, \cO^*_{X_{\bQ}})[1]).
\]
\end{lemma}
\begin{proof}
The first part follows by definition and the second one is a consequence of Lemma \ref{lem:picard-groupoid-in-terms-of-derived-category} and Remark \ref{remark:groupoids-and-cocones}. Here we use that for a map of complexes $A^{\bullet} \to B^{\bullet}$ we have that $\tau_{\leq 0}\cocone(A^{\bullet}\to B^{\bullet}) \simeq \tau_{\leq 0}\cocone(\tau_{\leq 0}A^{\bullet}\to \tau_{\leq 0}B^{\bullet})$.
\end{proof}
\begin{proposition} \label{prop:our-functors-under-thickenings} Let $Y$ be a finite type scheme over a Noetherian base scheme $S$ and let $f \colon Y \to X$ be a thickening in $Sch/S$. Then $\underline{\bG}_m(X) \otimes \bQ \simeq \underline{\bG}_m(Y) \otimes \bQ$ and $\underline{\PicS}(X) \otimes \bQ \simeq \underline{\PicS}(Y) \otimes \bQ$.
\end{proposition}
\begin{proof}
We deal with the case of $\underline{\PicS} \otimes \bQ$ as that of $\underline{\bG}_m(X) \otimes \bQ$ is analogous and simpler. Consider the following map of exact sequences
\begin{center}
\begin{tikzcd}
1 \arrow{r} & 1 + \cI \arrow{r} \arrow{d} & \cO^*_X \arrow{r} \arrow{d} & \cO^*_{Y} \arrow{r} \arrow{d} & 1 \\
1 \arrow{r} & 1 + \cI_{\bQ} \arrow{r} & \cO^*_{X_{\bQ}} \arrow{r} & \cO^*_{Y_{\bQ}} \arrow{r} & 1. 
\end{tikzcd}
\end{center}
\noindent By decomposing $f$ into square-zero extensions, we may assume that $\cI^2=0$. 

Applying $R\Gamma$ turns exact sequences into exact triangles and taking cocones in a map of exact triangles yields an exact triangle (\cite[Proposition 1.1.11]{BBDG}). Hence, we get 
\begin{align*}
\cocone( R\Gamma(X, 1 + \cI) \to R\Gamma(X_{\bQ}, 1 + \cI_{\bQ})) &\to \cocone(R\Gamma(X,\cO^*_X) \to R\Gamma(X_{\bQ},\cO^*_{X_{\bQ}}))\\
&\to \cocone(R\Gamma(Y, \cO^*_{Y}) \to R\Gamma(Y_{\bQ}, \cO^*_{Y_{\bQ}})) \xrightarrow{+1}.
\end{align*}
Let us call the leftmost term $C$. We claim that $C \otimes \bQ = 0$. Since tensoring by $\bQ$ is exact, this will imply that the two other terms are isomorphic up to tensoring by $\bQ$. Tensoring by $\bQ$ commutes with taking truncation as well, and so the proposition follows by Lemma \ref{lemma:kernel-of-Picard-groupoids-in-derived-categories}.

To show the claim, note that $(1+\cI) \simeq \cI$ as abelian sheaves, where the left term is endowed with the multiplicative structure and the right term is endowed with the additive structure. This follows from the fact that $\cI^2=0$ (use the formula $(1+a)(1+b) = 1 + (a+b)$ for local sections $a$ and $b$ of $\cI$). Therefore, 
\[
(1+\cI) \otimes \bQ \simeq 1 + \cI \otimes \bQ \simeq 1 + \cI_{\bQ} \simeq  1 + \cI_{\bQ} \otimes \bQ  \simeq (1 + \cI_{\bQ}) \otimes \bQ,
\]
and by exactness of tensoring by $\bQ$ over $\bZ$, the claim follows.\qedhere
\end{proof}
\begin{corollary}\label{cor:our-functors-under-thickenings-twisted}
Let $Y$ be a finite type scheme over a Noetherian base scheme $S$, let $M$ be a fixed line bundle on $S_{\bQ}$, and let $f \colon Y \to X$ be a thickening in $Sch/S$. Then $\underline{\PicS}_M(X) \otimes \bQ \simeq \underline{\PicS}_M(Y) \otimes \bQ$.

Similarly, $\underline{\bG}_{m,s}(X) \otimes \bQ \simeq \underline{\bG}_{m,s}(Y) \otimes \bQ$ for $s \in \bG_m(S_{\bQ})$.
\end{corollary}
\noindent This can be formally seen as follows. By the above proof, we get that
\begin{center}
\begin{tikzcd}
R\Gamma(X,\cO^*_X) \otimes \bQ \arrow{r} \arrow{d} & R\Gamma(Y,\cO^*_{Y}) \otimes \bQ \arrow{d}\\
R\Gamma(X_{\bQ},\cO^*_{X_{\bQ}}) \otimes \bQ \arrow{r} & R\Gamma(Y_{\bQ},\cO^*_{Y_{\bQ}}) \otimes \bQ
\end{tikzcd}
\end{center}
is a homotopy pullback square in the derived category of abelian groups regarded as a stable infinity category. The restriction to strictly symmetric monoidal groupoids corresponds, by Dold-Kan, to applying truncation to simplicial abelian groups. Since truncation is right adjoint, it preserves limits, which shows that the application of $\tau_{\leq 0}(-[1])$ keeps the diagram Cartesian. We can now conclude by Lemma \ref{lem:cartiesianity-checked-by-fibres}. Below, we give an elementary proof.
\begin{proof}
We start with the proof of the first part. By replacing $S$ by $X$, we can assume that $S=X$. Note that, by Lemma \ref{lem:cartiesianity-checked-by-fibres}, we are essentially showing that
\begin{center}
\begin{tikzcd}
\PicS(X) \arrow{r} \arrow{d} & \PicS(Y) \arrow{d}\\
\PicS(X_{\bQ}) \arrow{r} & \PicS(Y_{\bQ})
\end{tikzcd}
\end{center}
is Cartesian up to tensoring by $\bQ$. 

First, suppose that $\underline{\PicS}_M(Y) \otimes \bQ = \emptyset$. Then $\underline{\PicS}_M(X) \otimes \bQ = \emptyset$, and the statement of the corollary is proven. Thus, we may pick $(L_Y, \sigma) \in \underline{\PicS}_M(Y)$. We claim that it is enough to show that $(L_Y,M)$ is equal to the image of some line bundle $L' \in \pi_0(\PicS(X)) = \Pic(X)$ under the map 
\[
\Pic(X) \to \Pic(Y) \oplus \Pic(X_{\bQ}),
\]
up to replacing our line bundles by some multiples. Indeed, if that is the case, then tensoring by $L'$ and $f^*L'$ induce compatible isomorphisms of $\underline{\PicS}(X)$ with $\underline{\PicS}_M(X)$ and $\underline{\PicS}(Y)$ with $\underline{\PicS}_M(Y)$, respectively. Thus we can conclude by Proposition \ref{prop:our-functors-under-thickenings}.

To show the claim, note that the proof of Proposition \ref{prop:our-functors-under-thickenings}  gives that
\[
\cone(R\Gamma(X,\cO^*_X) \to R\Gamma(X_{\bQ},\cO^*_{X_{\bQ}})) \otimes \bQ \simeq \cone(R\Gamma(Y, \cO^*_{Y}) \to R\Gamma(Y_{\bQ}, \cO^*_{Y_{\bQ}})) \otimes \bQ,
\]
and so by Lemma \ref{lem:octahedral} we have the following exact triangle
\[
R\Gamma(X,\cO^*_X) \otimes \bQ \to R\Gamma(Y,\cO^*_{Y}) \otimes \bQ \oplus R\Gamma(X_{\bQ},\cO^*_{X_{\bQ}}) \otimes \bQ \to  R\Gamma(Y_{\bQ},\cO^*_{Y_{\bQ}}) \otimes \bQ \xrightarrow{+1}.
\]
This induces an exact sequence
\[
H^1(X,\cO^*_X) \otimes \bQ \to H^1(Y,\cO^*_{Y}) \otimes \bQ \oplus H^1(X_{\bQ},\cO^*_{X_{\bQ}}) \otimes \bQ \to  H^1(Y_{\bQ},\cO^*_{Y_{\bQ}}) \otimes \bQ,
\]
which identifies with
\[
\Pic(X) \otimes \bQ \to \Pic(Y) \otimes \bQ \oplus \Pic(X_{\bQ}) \otimes \bQ \to \Pic(Y_{\bQ}) \otimes \bQ.
\]
This concludes the proof as $(L_Y,M)$ lies in the kernel of the second map, and so is equal to the image of some $L' \in \Pic(X) \otimes \bQ$ under the first one.

The second part follows from the short exact sequence
\[
0 \to H^0(X,\cO^*_X) \otimes \bQ \to H^0(Y,\cO^*_{Y}) \otimes \bQ \oplus H^0(X_{\bQ},\cO^*_{X_{\bQ}}) \otimes \bQ \to  H^0(Y_{\bQ},\cO^*_{Y_{\bQ}}) \otimes \bQ
\]
which exists by the above paragraph.
\end{proof}
\noindent We emphasise here that we could not argue directly that $L'$ pullbacks to $(L_Y, M, \sigma)$, since it is not clear that for our chosen $L'$, the isomorphism $\sigma$ agrees, up to equivalence, with the canonical isomorphism $(L'|_Y)|_{Y_{\bQ}} \simeq (L'|_{X_{\bQ}})|_{Y_{\bQ}}$.

\begin{remark}
Corollary \ref{cor:our-functors-under-thickenings-twisted} was proven in \cite[Theorem 1.7]{witaszek2020keels} under some slightly different (but essentially equivalent) assumptions. In the old proof, however, it was necessary to leave the category of Noetherian schemes even when treating the simplest varieties. We believe that our new proof provides a better understanding of the considered problem. 
\end{remark}

\begin{lemma} \label{lemma:Gm-for-geometrically-connected-fibres} Let $f \colon Y \to X$ be a proper surjective morphism of Noetherian schemes with geometrically connected fibres. Set $s \in \bG_m(X_{\bQ})$. Then $f^* \colon \underline{\bG}_{m,s}(X) \otimes \bQ \to \underline{\bG}_{m,s}(Y) \otimes \bQ$ is an isomorphism.
\end{lemma}
\begin{proof}
If $f_*  \cO_Y = \cO_X$, then $f_* \cO^*_Y = \cO^*_X$ (here it is implicitly hidden that every invertible element of a ring has a unique inverse), and so $\bG_{m,s}(X)=\bG_{m,s}(Y)$. Hence, by Stein factorisation, we can assume that $f$ is a finite universal homeomorphism.

We may assume that $X$ and $Y$ are connected. Further, by Corollary \ref{cor:our-functors-under-thickenings-twisted}, we may assume that $X$ and $Y$ are reduced. If they are purely of positive characteristic $p>0$, then the result is standard ($f$ factors through a power of Frobenius, and so $\bG_m(X)[1/p] \simeq \bG_m(Y)[1/p]$). Thus, we may assume that $X_{\bQ} \neq \emptyset$, and $f$ is an isomorphism over some open subset of $X$. In particular, the conductors $C\subseteq X$ and $D \subseteq Y$ of $f$ are strict subsets (cf.\ Subsection \ref{ss:conductors}). By Lemma \ref{lemma:Milnor_Gm}, the following diagram is Cartesian
\begin{center}
		\begin{tikzcd}
			{\bG}_m(X) \arrow{r} \arrow{d} & {\bG}_m(Y) \arrow{d} \\
			{\bG}_m(C) \arrow{r} & {\bG}_m(D), 
		\end{tikzcd}
\end{center}
and so is (cf.\ Lemma \ref{lem:homotopy-fibre-of-t-stack-is-t-stack})
\begin{center}
		\begin{tikzcd}
			\underline{\bG}_{m,s}(X) \otimes \bQ \arrow{r} \arrow{d} & \underline{\bG}_{m,s}(Y) \otimes \bQ \arrow{d} \\
			\underline{\bG}_{m,s}(C) \otimes \bQ \arrow{r} & \underline{\bG}_{m,s}(D) \otimes \bQ.
		\end{tikzcd}
\end{center}
The morphism $D \to C$ is a finite universal homeomorphism, thus by Noetherian induction, we may assume that $\underline{\bG}_{m,s}(D) \otimes \bQ \simeq \underline{\bG}_{m,s}(C) \otimes \bQ$. Hence $\underline{\bG}_{m,s}(Y) \otimes \bQ \simeq \underline{\bG}_{m,s}(X) \otimes \bQ$.
\end{proof}


\begin{proposition}  \label{prop:Gm-blow-up-squares}
Let $f \colon Y \to X$ be a proper morphism of Noetherian schemes which is an isomorphism outside of a closed subset $Z \subseteq X$ with preimage $E \subseteq Y$. Set $s \in \bG_m(X_{\bQ})$. Then
\begin{center}
		\begin{tikzcd}
			\underline{\bG}_{m,s}(X) \arrow{r} \arrow{d} & \underline{\bG}_{m,s}(Y) \arrow{d} \\
			\underline{\bG}_{m,s}(Z) \arrow{r} & \underline{\bG}_{m,s}(E),
		\end{tikzcd}
\end{center}
is Cartesian up to tensoring by $\bQ$.
\end{proposition}
\begin{proof}
By Stein factorisation and Lemma \ref{lemma:Gm-for-geometrically-connected-fibres}, we may assume that $f$ is a finite morphism (here we use that the restriction of the connected fibres part of the Stein factorisation of $Y \to X$ to $E$ has geometrically connected fibres).

Let $\cI$ be the ideal defining $Z \subseteq X$ and Let $Z_n$ be the subscheme defined by $\cI^n$ for $n \in \bN$. By Corollary \ref{cor:our-functors-under-thickenings-twisted} we may assume that $X$ and $Y$ are reduced, and we can replace $Z$ by $Z_n$ and $E$ by the scheme theoretic pullback of $Z_n$. Hence the above diagram is Cartesian for $\bG_m$ itself by Lemma \ref{lemma:Milnor_Gm2}, and so it is Cartesian for $\underline{\bG}_{m,s} \otimes \bQ$ as well (cf.\ the argument of Lemma \ref{lem:homotopy-fibre-of-t-stack-is-t-stack}). \qedhere
\end{proof}

\begin{theorem} \label{thm:fibre-of-Gm-is-h-sheaf}
Let $S$ be a Noetherian scheme and let $Sch/S$ be the category of schemes of finite type over $S$. Fix $s \in \bG_m(S_{\bQ})$. 
Then $\underline{\bG}_{m,s} \otimes \bQ$ is a sheaf for the h-topology.
\end{theorem} 
\begin{proof}
This follows from Theorem \ref{thm:BS-criterion-for-h-sheaves} and Proposition \ref{prop:Gm-blow-up-squares} as $\bG_{m,s}$ (and so $\underline{\bG}_{m,s} \otimes \bQ$) is a sheaf for the fppf topology (cf.\ Lemma \ref{lem:homotopy-fibre-of-t-stack-is-t-stack}).
\end{proof}

\begin{proposition} \label{prop:fibre-of-Pic-blow-up-squares}
Let $X$ be a finite type scheme over a Noetherian base scheme $S$ and let $f \colon Y \to X$ be a proper map which is an isomorphism outside of a closed subset $Z \subseteq X$ with preimage $E \subseteq Y$. Fix a line bundle $M$ on $S_{\bQ}$. Then
\begin{center}
		\begin{tikzcd}
			\underline{\PicS}_M(X) \arrow{r} \arrow{d} & \underline{\PicS}_M(Y) \arrow{d} \\
			\underline{\PicS}_M(Z) \arrow{r} & \underline{\PicS}_M(E),
		\end{tikzcd}
\end{center}
is Cartesian up to tensoring by $\bQ$.
\end{proposition}
\begin{proof}
By replacing $S$ by $X$, we may assume that $S=X$.

The full faithfullness of $\underline{\PicS}_M(X)\to \underline{\PicS}_M(Y)  \times_{\underline{\PicS}_M(E)} \underline{\PicS}_M(Z)$ up to tensoring by $\bQ$ follows from Proposition \ref{prop:Gm-blow-up-squares} applied to the Hom-bundles. More precisely, we need to check that for every $\cL_1, \cL_2 \in \underline{\PicS}_M(X)$ the functor $\cF \colon W \mapsto \Hom(\pi^*\cL_1, \pi^*\cL_2)$ on schemes over $X$ (here $\pi$ is the projection to $X$) satisfies that $\cF(X) \to \cF(Y) \times_{\cF(E)} \cF(Z)$ is an isomorphism up to tensoring by $\bQ$. Let $\cL_i = (L_i, \phi_i)$ for a line bundle $L_i$ on $X$ and an isomorphism $\phi_i \colon L_i|_{X_{\bQ}} \xrightarrow{\simeq} M$. It is enough to check the above assertion affine locally, so we can assume that $L_1$ and $L_2$ are trivial. Then $\cF$ is isomorphic to $\underline{\bG}_{m,s}$, where $s = \phi^{-1}_2 \circ \phi_1 \in \bG_m(X_{\bQ})$. Hence, Proposition \ref{prop:Gm-blow-up-squares} applies here.\\

Therefore, we just need to verify that $\underline{\PicS}_M(X) \to \underline{\PicS}_M(Y) \times_{\underline{\PicS}_M(E)} \underline{\PicS}_M(Z)$ is essentially surjective up to tensoring by $\bQ$. First, we assume that $Y$ is the blow-up of $X$ along $Z$. By Zariski descent, we may further assume that $X$ is affine and by flatness of completions and Beauville-Laszlo-type descent (see \cite[Tag 0AF0]{stacks-project}) that it is complete with respect to the ideal of $Z$. By Corollary \ref{cor:our-functors-under-thickenings-twisted}, we may replace $Z$ and $E$ by their thickenings, so that $\underline{\PicS}_M(X) \to \underline{\PicS}_M(Y) \times_{\underline{\PicS}_M(E)} \underline{\PicS}_M(Z)$ is essentially surjective up to tensoring by $\bQ$ by Lemma \ref{lemma:blow-up-square-Pic} (cf.\ the argument of Lemma \ref{lem:homotopy-fibre-of-t-stack-is-t-stack}). This concludes the proof of the proposition in this case. \\

In what follows, we reduce to the case of blow-up squares. By \cite[Corollary 5.7.12]{rg71} (also \cite[Tag 081T]{stacks-project}), there exists a blow-up $r \colon W \to X$ along a closed subscheme $Z'$  disjoint from $X \, \backslash\, Z$ which admits a factorisation through $g \colon W \to Y$. We may assume that $\Supp Z' = \Supp Z$ (see for example \cite[Tag 080A]{stacks-project}), and so that $Z'=Z$  by Corollary \ref{cor:our-functors-under-thickenings-twisted}. Set $G = g^{-1}(E)$.

Take $h \colon T \to Y$ to be the blow-up of $E=f^{-1}(Z)$. By the universal property of blow-ups over $Y$ (\cite[Tag 0806]{stacks-project}), the morphism $g \colon W \to Y$  factorises through $T$ yielding $W \xrightarrow{g'} T \xrightarrow{h} Y$.
We claim that $g' \colon W \to T$ is a thickening. 
\begin{center}
\begin{tikzcd}
W \arrow{r}{g'} \arrow[bend left = 35]{rr}{g} & T \arrow[dashed, bend left = 45]{l} \arrow{r}{h} & Y \arrow{r}{f} &  X.
\end{tikzcd}
\end{center} 
Indeed, by the universal property of blow-ups over $X$, we get a map $T \to W$ which composed $W \to T \to W$ with $g'$ yields an isomorphism (by the universal property of blow-ups over $X$ again). Hence $W \to T$ is a proper monomorphism, and so a closed embedding by \cite[Theorem 18.12.6]{EGAIV}. Since $g' \colon W \to T$ is an isomorphism outside of a Cartier divisor $h^{-1}(E)$, this means that $g'$ is dominant (\cite[Tag 07ZU]{stacks-project}), and so a thickening. 

By the above paragraph and Corollary \ref{cor:our-functors-under-thickenings-twisted}, we get that $\underline{\PicS}_M(W) \simeq \underline{\PicS}_M(T)$ and $\underline{\PicS}_M(G) \simeq \underline{\PicS}_M(h^{-1}(E))$ up to tensoring by $\bQ$. Now, in the following diagram
\begin{center}
\begin{tikzcd}
\underline{\PicS}_M(G)   &  \arrow{l} \underline{\PicS}_M(W)  \\
\underline{\PicS}_M(E)  \arrow{u} & \arrow{l} \underline{\PicS}_M(Y)  \arrow{u} \\ 
\underline{\PicS}_M(Z) \arrow{u} & \arrow{l} \underline{\PicS}_M(X), \arrow{u}
\end{tikzcd}
\end{center}
the big outer square and the upper square are Cartesian up to tensoring by $\bQ$ by the case of blow-up squares. Therefore, the lower  square is Cartesian up to tensoring by $\bQ$ by the $2$-out-of-$3$ property for fibre squares (a repeated application of \cite[Tag 02XD]{stacks-project}, see also Lemma \ref{lem:cartiesianity-checked-by-fibres}). \qedhere
\end{proof}
We warn the reader that in the proof of the above theorem it is essential to consider blow-ups which are not surjective. For example, if $Y= \bP^1 \cup \bP^1 \subseteq \bP^1 \times \bP^1$ is the union of the standard coordinate lines and $Y \to X = \bP^1$ is the projection onto the first factor, then $W \to X$ is the blow-up at $0 \in \bP^1$ which is just the identity. In this case, $W = T \to Y$ is a closed embedding which is the blow-up of $Y$ along the second irreducible component. 

\begin{proof}[Proof of Theorem \ref{thm:fibre-of-Pic-is-h-sheaf}]
This follows from Theorem \ref{thm:BS-criterion-for-h-sheaves} and Proposition \ref{prop:fibre-of-Pic-blow-up-squares} as $\PicS$ (and so $\underline{\PicS}_M \otimes \bQ$) is a stack in groupoids for the fppf topology (see Lemma \ref{lem:homotopy-fibre-of-t-stack-is-t-stack}).
\end{proof}

This gives a new proof of the following result.
\begin{theorem}[{\cite[Theorem 1.7]{witaszek2020keels}}]\label{thm:Pic-under-uh}
Let $f \colon Y \to X$ be a finite universal homeomorphism of Noetherian schemes. Then $\underline{\PicS}_M(X) \otimes \bQ \simeq \underline{\PicS}_M(Y) \otimes \bQ$ for every line bundle $M$ on $X_{\bQ}$ and the following diagram 
\begin{center}
\begin{tikzcd}
\PicS(X) \arrow{r}{f^*} \arrow{d} & \PicS(Y) \arrow{d} \\
\PicS(X_{\bQ}) \arrow{r} & \PicS(Y_{\bQ}),
\end{tikzcd}
\end{center}
is Cartesian up to tensoring by $\bQ$. 
\end{theorem} 
\begin{proof}
Since $Y\times_X Y$ and $Y\times_X Y \times_X Y$ are thickenings of $Y$, we have that $\underline{\PicS}_M(X) \otimes \bQ \simeq \underline{\PicS}_M(Y) \otimes \bQ$ by Theorem \ref{thm:fibre-of-Pic-is-h-sheaf} and Corollary \ref{cor:our-functors-under-thickenings-twisted} for every $M \in \Pic(X_{\bQ})$. Hence, the theorem holds by Lemma \ref{lem:cartiesianity-checked-by-fibres}.
\end{proof}

\section{The proof of the main theorem}
The goal of this section is to prove Theorem \ref{thm:main-intro}. As explained in the introduction, the proof is split into three parts.

\subsection{Relatively trivial case}

\begin{proposition} \label{prop:PicS-ff}
Let $X$ be a scheme admitting a proper surjective morphism with geometrically connected fibres $\pi \colon X \to S$ to an excellent scheme $S$. Let $M$ be a line bundle on $S_{\bQ}$. Then $\pi^* \colon \underline{\PicS}_M(S) \otimes \bQ \to \underline{\PicS}_M(X) \otimes \bQ$ is fully faithful with the essential image consisting of all pairs $(L,\phi)$, where $L$ is a line bundle on $X$ such that $L|_{X_s}$ is torsion for every point $s \in S$ having positive characteristic residue field, and $\phi \colon L|_{X_{\bQ}} \simeq \pi_{\bQ}^*M$ is an isomorphism.
\end{proposition}
\noindent We follow the ideas of \cite[Tag 0EXH]{stacks-project} and repeatedly use Theorem \ref{thm:fibre-of-Pic-is-h-sheaf}.
\begin{proof}
The full faithfullness of $\pi^* \colon \underline{\PicS}_M(S) \to \underline{\PicS}_M(X)$ up to tensoring by $\bQ$ means that the map $\Hom(\cL_1,\cL_2) \to \Hom(\pi^*\cL_1, \pi^*\cL_2)$ is an isomorphism up to tensoring by $\bQ$ for every $\cL_1, \cL_2 \in \underline{\PicS}_M(S)$. Let $\cL_i = (L_i, \phi_i)$ for a line bundle $L_i$ on $S$ and an isomorphism $\phi_i \colon L_i|_{S_{\bQ}} \xrightarrow{\simeq} M$. It is enough to check the above assertion affine locally, so we can assume that $L_1$ and $L_2$ are trivial. Then the above map identifies with $\underline{\bG}_{m,s}(S) \otimes \bQ \to \underline{\bG}_{m,s}(X) \otimes \bQ$, where $s = \phi^{-1}_2 \circ \phi_1 \in \bG_m(X_{\bQ})$. This is an isomorphism by Lemma \ref{lemma:Gm-for-geometrically-connected-fibres}.


Hence, it is enough to show that every $\cL \in \underline{\PicS}_M(X) \otimes \bQ$ as in the statement of the proposition lies in the essential image of $\pi^*$. By Corollary \ref{cor:our-functors-under-thickenings-twisted}, we may assume that $X$, and so also $S$, are reduced. Further, as the finite part of the Stein factorisation is a universal homeomorphism, we may assume that $\pi_* \cO_X = \cO_S$ by Theorem \ref{thm:Pic-under-uh}.

Since $\pi_*\cO_X=\cO_S$, we have that $\pi$ is flat over an open dense subset $U$ of $S$ (\cite[Tag 052B]{stacks-project}), and the flattening (\cite[Tag 081R]{stacks-project}) provides us with a $U$-admissible blow-up $g \colon S' \to S$ (with exceptional locus $E$) such that the strict transform $X'$ in $X \times_S S'$ is flat over $S$. Since $U$ is dense, no irreducible components of $S$ are contained in $S \,\backslash\, U$. In particular, $g$ induces a bijection between irreducible components of $S'$ and $S$ (\cite[Tag 0BFM]{stacks-project}), and so $\pi' \colon X' \to S'$ is generically a contraction (that is, $\pi'_*\cO_{X'} = \cO_{S'}$ holds over all generic points of $S'$). As each point on $E$ is a specialisation of a point on $S' \, \backslash \, E$, we have that $\pi' \colon X' \to S'$ has geometrically connected fibres by \cite[Tag 0BUI]{stacks-project}. Consider the following Cartesian diagram
\begin{center}
\begin{tikzcd}
X' \cup E \times_S X \arrow{r}{g_X} \arrow{d}{\pi'} & X \arrow{d}{\pi} \\
S' \arrow{r}{g} & S.
\end{tikzcd}
\end{center}
We claim that it is enough to show the proposition for $g_X^*\cL|_{X'}$ over $S'$. By Lemma \ref{lem:essential-image-h-topology}, it is enough to show the proposition for $g_X^*\cL$ over $S'$, and the claim holds by Lemma \ref{lem:essential-image-h-topology2} and Noetherian induction (applied to $E \times_S X \to E$, which has geometrically connected fibres; also note that $\pi'(X'\cap (E\times_X X)) = E$ and $X' \cap (E \times_S X) = (\pi'|_{X'})^{-1}(E)$, thus $\pi'|_{X' \cap (E \times_S X)}$ has geometrically connected fibres). Here, we use that pullback maps on $\underline{\PicS}_M \otimes \bQ$ for all proper surjective morphisms with geometrically connected fibres are fully faithful as proven in the first paragraph.

By replacing $X \to S$ by $X' \to S'$, we may thus assume that $X \to S$ is flat. By the same argument, we may replace $S$ by its normalisation $\tilde S \to S$, and $X$ by the base change $X \times_S \tilde S$, and assume that $S$ is normal. We know that $X \to S$ has geometrically connected fibres and, over a dense open subset, is a contraction. This means that the finite part of the Stein factorisation is birational (isomorphic over a dense open subset) and, since $S$ is normal, it must be an isomorphism. Hence, $\pi_*\cO_X = \cO_S$. 

Let $\cL = (L, \phi)$, where $L$ is the underlying line bundle. It is enough to show that $L^m \simeq \pi^*E$ for  some line bundle $E$ on $S$ and $m \in \bN$. Indeed, $(\pi_{\bQ})_*\cO^*_{X_{\bQ}} = \cO^*_{S_{\bQ}}$, and so $\phi \in \Hom(L|_{X_{\bQ}}, \pi_{\bQ}^*M)$ must then descend uniquely to an isomorphism in $\Hom(E|_{S_{\bQ}}, M)$ (this can be checked locally in which case we can assume that $E$ and $M$ are trivial). 

Since $L$ is semiample over every point of $S$ (and so also over the generic points), there exists a \emph{dense} open subset $U \subseteq S$ such that $L|_{X_U}$ is semiample over $U$, where $X_U = X \times_S U$. Let $\eta_1, \ldots, \eta_k$ be the generic points of these irreducible components of $Z := S \,\backslash\, U$ which are of codimension one. Then the localisations $A_i := \cO_{X,\eta_i}$ at $\eta_i$ are divisorial valuation rings. We may assume that $S_{\bQ} \subset U$, and so the residue fields of $A_i$ are of positive characteristic.

By \cite[Tag 0EXG]{stacks-project}, it is enough to show that $L^m|_{X_{i}}$ is trivial over $\Spec A_i$ for $1 \leq i \leq k$ and some $m \in \bN$, where $X_i = X \times_S \Spec A_i$. In turn, to show this statement, it is enough to prove that for every $n \in \bN$ there exists $m \in \bN$ such that $L^m|_{X_{n,i}}$ is trivial, where $X_{n,i}$ is the zero locus in $X_i$ of the $n$-th power of the uniformiser of $A_i$ (\cite[Tag 0EXF]{stacks-project}). Since $L$ is semiample over the special point of $\Spec A_i$, this follows from Theorem \ref{thm:universal-homeomorphism-semiample}, concluding the proof of the proposition. 
\end{proof}
\noindent We point out to the reader that it is important to consider only a finite number of codimension one points. Otherwise, it could a priori happen that the power of $L$ that renders $L|_{X_i}$ trivial is unbounded, and so we would not be able to apply \cite[Tag 0EXG]{stacks-project}.

\begin{proposition} \label{prop:relatively-trivial-case}
Let $X$ be a scheme admitting a proper morphism $\pi \colon X \to S$ with geometrically connected fibres to an excellent scheme $S$. Let $L$ be a line bundle on $X$. Then $L$ is relatively torsion if and only if some multiple of $L|_{X_{\bQ}}$ descends to $S_{\bQ}$ and $L|_{X_s}$ is relatively torsion for every point $s \in S$ having positive characteristic residue field.
\end{proposition}
\begin{proof}
This is immediate from Proposition \ref{prop:PicS-ff}.
\end{proof}

We also get a variant for algebraic spaces.
\begin{proposition} \label{prop:relatively-trivial-case-algebraic-spaces}
Let $X$ be a scheme admitting a proper morphism $\pi \colon X \to S$ to an excellent algebraic space $S$ such that $\pi_*\cO_X=\cO_S$. Let $L$ be a line bundle on $X$. Then $L$ descends to $S$ if and only if some multiple of $L|_{X_{\bQ}}$ descends to $S_{\bQ}$ and $L|_{X_s}$ is relatively torsion for every point $s \in S$ having positive characteristic residue field.
\end{proposition}
\begin{proof}
We need to show that $\pi_*L^m$ is a line bundle for some $m>0$ which can be verified \'etale locally. Hence, we can assume that $S$ is a scheme, and so we can conclude by Proposition \ref{prop:relatively-trivial-case}.
\end{proof}

\subsection{Normal case}
We need the following lemma.
\begin{lemma} \label{lem:pullback-via-alteration}
Let $X$ be a normal integral scheme admitting a proper surjective morphism $\pi \colon X \to S$ to an excellent integral scheme $S$ satisfying $\dim S < \dim X$. Let $L$ be a relatively nef line bundle on $X$ such that $L|_{X_{\eta}}$ is trivial, where $\eta$ is the generic point of $S$, $L|_{X_s}$ is semiample for every positive characteristic point $s\in S$, and $L|_{X_{\bQ}}$ is relatively semiample. Let $g \colon S' \to S$ and $g_X \colon X' \to X$ be proper, surjective, and generically finite maps sitting in the following diagram
\begin{center}
\begin{tikzcd}
X' \arrow{r}{g_X} \arrow{d}{\pi'} & X \arrow{d}{\pi} \\
S' \arrow{r}{g} & S,  
\end{tikzcd}
\end{center}
where $X'$ and $S'$ are integral.

Suppose that $L' := g_X^*L$ is semiample over $S'$ and that Theorem \ref{thm:main-intro} holds for schemes of dimension smaller than $\dim X$. Then $L$ is semiample over $S$.
\end{lemma}
\begin{proof}

Let $X' \xrightarrow{h_1} W \xrightarrow{h_2} S'$ be a factorisation of $\pi'$ with $h_1$ being the projective contraction associated to $L'$ over $S'$. Since $L'$ is trivial over the generic point of $S'$, the morphism $h_2 \colon W \to S'$ is generically finite. In particular, the induced morphism $W \to S$ is generically finite, and so $\dim W = \dim S < \dim X$. Further, up to replacing line bundles by some multiples, $L' \sim h_1^*M$ for a line bundle $M$ on $W$.

We claim that $M$ satisfies the assumptions of Theorem \ref{thm:main-intro} over $S$, that is $M|_{W_{\bQ}}$ is semiample over $S$ and $M|_{W_s}$ is semiample for every positive characteristic $s \in S$ and the fibre $W_s$ of $W \to S$ over $s$. For the former, note that $L'|_{X_{\bQ}}$ is semiample over $S_{\bQ}$ by Lemma \ref{lem:semiample-under-pullback}, and so  $M|_{W_{\bQ}}$ is semiample by the same reference in view of $h_{1,\bQ} \colon X'_{\bQ} \to W_{\bQ}$ being a contraction.   The latter follows, for example, by Corollary \ref{cor:connected-fibres-semiample} (in fact we just need a positive characteristic variant thereof which is much simpler) as $X'_s \to W_s$ has geometrically connected fibres and $L'|_{X'_s} = (h_1)^*M|_{X_s}$ is semiample. 

In particular, $M$ is semiample over $S$, and so are $L'$ and $L$ by Lemma \ref{lem:semiample-under-pullback} (here we use the normality of $X$).
\end{proof}

We first show the following variant of Proposition \ref{prop:relatively-trivial-case}.
\begin{proposition} \label{prop:normal-case-partial}
Fix a natural number $n\in \bN$ and suppose that Theorem \ref{thm:main-intro} holds for all schemes $X$ of dimension at most $n-1$. Then the statement of Theorem \ref{thm:main-intro} holds in dimension $n$ when, in addition, $X$ is normal integral, $L|_{X_{\eta}} \sim_{\bQ} 0$, and $\dim S < \dim X$. Here $\eta$ is the generic point of $\pi(X)$.
\end{proposition}
\begin{proof}
We can assume that $S$ is integral, and $\pi \colon X \to S$ is a surjective contraction (cf.\ \cite[Lemma 2.10]{ct17}). First, we reduce to the case of $\pi$ being equidimensional. By the generic flatness (\cite[Tag 052A]{stacks-project}), $\pi$ is flat over an open dense subset of $S$. Thus, we can apply a flattening (\cite[081R]{stacks-project}), to get a projective birational map $g \colon S' \to S$ such that $X' \to S'$ is flat  where $X'$ is the irreducible component of $X \times_S S'$ dominant over $S'$. By replacing $X'$ by its normalisation, we may assume that $X'$ is normal; this might break the flatness of $\pi'$ but it preserves its equidimensionality. In order to show the proposition we may replace $X \to S$ by $X' \to S'$ thanks to Lemma \ref{lem:pullback-via-alteration}, so that we can assume $X \to S$ to be equidimensional. A priori $X \to S$ is now only generically a contraction.

Analogously, we may replace $X \to S$ by any $X' \to S'$ such that $S' \to S$ is a proper, surjective, generically finite map and $X'$ is the normalisation of the irreducible component of $X \times_S S'$ dominant over $S'$ (Lemma \ref{lem:pullback-via-alteration}), or such that $S' \to S$ is \'etale surjective and $X' = X \times_S S'$ (Lemma \ref{lem:semiample-under-faithfully-flat-cover}). Here we replace $L$ by its pullback to $X'$. Note that $X' \to S'$ is equidimensional, and so every such replacement preserves the equidimensionality of fibres. 

 In particular, by Theorem \ref{thm:regular-alteration} applied to $S$, we may assume that $S$ is regular and $X \to S$ is equidimensional. Further, as $X \to S$ is generically a contraction, the finite part of the Stein factorisation is birational, and since $S$ is normal, it must be an identity. Hence, $X \to S$ is a contraction. Moreover, the assumption that $L|_{X_{\eta}} \sim_{\bQ} 0$ for the generic point $\eta \in S$ is also preserved. 
 Hence the proposition follows from Lemma \ref{lem:descend-equidimensional-base-qfactorial}.



\end{proof}

\begin{proposition} \label{prop:normal-case}
Fix a natural number $n\in \bN$ and suppose that Theorem \ref{thm:main-intro} holds for all schemes $X$ of dimension at most $n-1$. Then the statement of Theorem \ref{thm:main-intro} holds in dimension $n$ when $X$ is normal.
\end{proposition}
\begin{proof}
We can assume that $S$ is normal integral, $\pi \colon X \to S$ is surjective, and $X$ is normal integral (cf.\ \cite[Lemma 2.10]{ct17}). Further, by Chow's lemma (\cite[Tag 088U]{stacks-project}), we may assume that $X$ is projective over $S$. Indeed, if the pullback of $L$ under a surjective proper map is semiample, then so is $L$ itself in view of $X$ being normal integral (Lemma \ref{lem:semiample-under-pullback}). Let $\eta$ be the generic point of $S$. If $L$ is big (equivalently, $L|_{X_{\eta}}$ is big), then by Theorem \ref{theorem:mixed-char-Keel}, it is enough to verify that $L|_{\mathbb{E}(L)}$ is semiample, which is the case by assumptions as $\dim \mathbb{E}(L) \leq n-1$. Thus we may assume that $L$ is not big. 

By assumptions, $L|_{X_{\eta}}$ is semiample, and so there exists a fibration $f_{\eta} \colon X_{\eta} \to Z_{\eta}$ such that $L|_{X_{\eta}}$ is torsion over $Z_{\eta}$. 
Since $L$ is not big, we have that $\dim Z_{\eta} < \dim X_{\eta}$. 

Let $Z$ be a proper compactification of $Z_{\eta}$ over $S$ which exists by the Nagata compactification (see \cite[Tag 0F41]{stacks-project}; here we spread out $f_{\eta}$ to be defined over an open subset of $S$, so that the schemes in question are of finite type). We may assume that $Z$ is normal. By resolving the indeterminacies of $X \dashrightarrow Z$, we get projective surjective maps $h_1 \colon W \to X$ and $h_2 \colon W \to Z$ with $W$ normal. By replacing $Z$ by the finite part of the Stein factorisation of $h_2 \colon W \to Z$, we may assume that $h_2$ is a contraction. By applying Proposition \ref{prop:normal-case-partial} to $X$, $L$, $S$ replaced by $W$, $h_1^*L$, $Z$, respectively, we get that $h_1^*L^m \sim h_2^*M$ for some line bundle $M$ on $Z$ and $m \in \bN$. 

Since $h_2|_{\bQ} \colon W_{\bQ} \to Z_{\bQ}$ is a contraction, $M|_{Z_{\bQ}}$ is semiample by Lemma \ref{lem:semiample-under-pullback}. Moroever, as $W_s \to Z_s$ has geometrically connected fibres for every positive characteristic $s \in S$ and the fibres $W_s$, $Z_s$ over it, we have that  $M|_{Z_s}$ is semiample (cf.\ Corollary \ref{cor:connected-fibres-semiample}).
By the assumption on the validity of Theorem \ref{thm:main-intro} in lower dimensions, $M$ is relatively semiample over $S$, thus so is $h_1^*L$ and hence $L$ (see Lemma \ref{lem:semiample-under-pullback}, here $X$ is normal).
\end{proof}

\subsection{Non-normal case}
We can conclude the proof.
\begin{proof}[Proof of Theorem \ref{thm:main-intro}]
We argue by induction on dimension. 
By Theorem \ref{thm:universal-homeomorphism-semiample} we may assume that $X$ and $S$ are reduced.


Let $g \colon Y \to X$ be the normalisation. By Lemma \ref{lem:semiample-under-pullback}, $g^*L$ is semiample on each fibre over $S$ and $g^*L|_{Y_{\bQ}}$ is semiample. Hence Proposition \ref{prop:normal-case} implies that $g^*L$ is semiample over $S$. Let $h \colon Y \to Z$ be the associated semiample fibration. 

Let $E = (Y \times_X Y)_{\red} \rightrightarrows Y$ be a finite set theoretic equivalence relation. Note that $E \subseteq Y \times_S Y$ as $X \to S$ is separated (\cite[Tag 01KR]{stacks-project}). Set $E_{Z} = (h\times h)(E) \subseteq Z \times_S Z$. Below, we prove two claims.

\begin{claim} There exists a closure $\overline E_{Z}$ of $E_{Z}$ as a finite set theoretic equivalence relation. 
\end{claim}
\begin{proof}
To this end, we consider a factorisation $\emptyset = S_0 \subseteq S_1 \subseteq \ldots \subseteq S_m = S$ of closed subschemes of $S$ such that $L|_{X_{i}}$ is semiample for all $i\geq 1$ where $U_i := S_i \, \backslash \, S_{i-1}$ and $X_{i} = X \times_S U_i$. Such a factorisation is constructed as follows. Pick a generic point $\eta \in S$. Since $L|_{X_{\eta}}$ is semiample, there exists an open subset $U \subseteq S$ such that $L|_{X \times_S U}$ is semiample over $U$. Then the factorisation is constructed by applying Noetherian induction to $X \times_S (S \,\backslash\, U)$ (so that $S_{m-1} = S \,\backslash\, U$).

Set $Y_{i} = Y \times_S U_i$, $Z_{i} = Z \times_S U_i$, $E_{i} = E \times_S U_i$, and $E_{Z_i} = E_Z \times_S U_i$. Here, $E_{Z_i} \subseteq Z_i \times_S Z_i$ since closed immersions are stable under base change (\cite[Tag 03M4]{stacks-project}). Note that $Y_i \to Z_i$ has geometrically connected fibres. Let $Y_i \to Z'_i \to Z_i$ be the Stein factorisation of $Y_i \to Z_i$, with $Z'_i \to Z_i$ being a universal homeomorphism. Further, let $f_{i} \colon X_{i} \to W_{i}$ be the semiample fibration associated to $L|_{X_{i}}$. We have a factorisation $Y_{i}  \to Z'_{i} \xrightarrow{\nu_i} W_{i}$, where $\nu_i$ is finite. Last, we define $E_{Z'_i}$ to be the image of $E_i$ in $Z'_i \times_S Z'_i$.
\begin{center}
\begin{tikzcd}
E_i \arrow{d} \arrow[shift left = 1]{r} \arrow[shift right = 1]{r} & Y_i \arrow{d} \arrow{r} & X_i \arrow{d} \\
E_{Z'_i} \arrow{d}  \arrow[shift left = 1]{r} \arrow[shift right = 1]{r} & Z'_i \arrow{d} \arrow{r}{\nu_i} & W_i \\
E_{Z_i} \arrow[shift left = 1]{r} \arrow[shift right = 1]{r} & Z_i.  
\end{tikzcd}
\end{center}
Note that $E_{i} \rightrightarrows Y_i \to W_i$ is a coequaliser diagram, and so $E_i \rightrightarrows Z'_i \to W_i$
is also a coequaliser diagram. Since $E_{Z'_i}$ is an image of $E_i$, the pullback $\cO_{E_{Z'_i}} \to \cO_{E_i}$ is injective, and
\[
E_{Z'_i} \rightrightarrows Z'_i \to W_i
\]
is a coequaliser diagram. As a consequence, $E_{Z'_i}$ is contained in the finite set theoretic equivalence relation $Z'_i \times_{W_i} Z'_i$, and so the reduction of  $E_{Z_i}$ is also contained in a finite set theoretic equivalence relation being the image of $Z'_i \times_{W_i} Z'_i$ under the universal homeomorphism $Z'_i \times_{S} Z'_i \to Z_i \times_S Z_i$.
Therefore, the assumptions of Lemma \ref{lem:closures-of-subsets-of-finite-equivalence-relations} are satisfied and the claim is proven.
\end{proof}
\begin{claim}The quotient of $Z$ by $\overline E_Z$ exists as a separated algebraic space of finite type over $S$.
\end{claim}
\begin{proof} By Theorem \ref{thm:quotients}, it is enough to show that such a quotient $Z_{\bQ}/{\overline E_{Z,\bQ}}$ exists in characteristic zero. To this end, take $X_{\bQ} \to W_{\bQ}$ to be the semiample fibration associated to $L|_{X_{\bQ}}$. We have a factorisation $Y_{\bQ} \to Z_{\bQ} \xrightarrow{\nu_{\bQ}} W_{\bQ}$ where $\nu_{\bQ}$ is finite and 
\[
E_{Z,\bQ} \rightrightarrows Z_{\bQ} \xrightarrow{\nu_{\bQ}} W_{\bQ}
\]
is coequaliser diagram. Indeed, $E_{\bQ} \rightrightarrows Z_{\bQ} \to W_{\bQ}$ is a coequaliser diagram and the statement follows as $E_{Z,\bQ}$ is an image of $E_{\bQ}$, and so the pullback $\cO_{E_{Z,\bQ}} \to \cO_{E_{\bQ}}$ is injective. Hence, 
\[
\overline{E}_{Z,\bQ} \rightrightarrows Z_{\bQ} \xrightarrow{\nu_{\bQ}} W_{\bQ}
\]
is a coequaliser diagram, and the quotient exists by Proposition \ref{prop:kollar_quotients_exist}.
\end{proof}

Let $W = Z/ {\overline E_Z}$ be the quotient and consider the following diagram
\begin{center}
\begin{tikzcd}
Y\arrow{dr} \arrow{dd}[swap]{h} \arrow[bend left = 5]{drr}{g}  & & \\
& X^* \arrow{d}{f} \arrow{r}{r} & X  \\
Z \arrow{r}{\nu} & W  &  
\end{tikzcd}
\end{center} 
where $X^*$ is the image of $Y$ in $W \times_S X$. By construction, $r$ is a finite universal homeomorphism. Since $h \colon Y \to Z$ is the Stein factorisation of $Y \to W$, we can replace $W$ by the finite part of the Stein factorisation of $X^* \to W$ so that $f_* \cO_{X^*} = \cO_{W}$ and $Z$ still admits a morphism to $W$ which we also call $\nu$.  

By Lemma \ref{lem:semiample-under-pullback}, we have that $r^*L|_{X^*_{\bQ}}$ is semiample (and hence descends to $W_{\bQ}$ as it is relatively numerically trivial) and that $r^*L$ is semiample on fibres over positive characteristic points of $S$, and so also on fibres over each positive characteristic point of $W$. Since $r^*L$ is relatively numerically trivial over $W$, it is torsion on each positive characteristic fibre over $W$. Therefore, $r^*L^m \sim f^*M$ for some $m \in \bN$ and a line bundle $M$ on $W$ by Proposition \ref{prop:relatively-trivial-case-algebraic-spaces}. Note that the pullback of $M$ to $Y$ is isomorphic to $g^*L^m$, which means that $\nu^*M$ is isomorphic to the ample line bundle induced by $g^*L^m$ up to replacing $m$ and $M$ by some multiples. If a pullback of a line bundle by a finite surjective map is ample, then so is the line bundle itself (this follows, for example, by the Nakai-Moishezon criterion, cf.\ \cite[Lemma 2.15]{BMPSTWW20}). Therefore, $M$ is ample, and so is $r^*L^m$. Finally $L$ is semiample by Corollary \ref{cor:connected-fibres-semiample} as $r$ is a universal homeomorphism. 

\end{proof}


\section*{Acknowledgements}
We are grateful to Bhargav Bhatt for numerous helpful conversations. We also thank Piotr Achinger, Paolo Cascini, Elden Elmanto, Christopher Hacon, Akhil Mathew, and Hiromu Tanaka for comments and suggestions.


\bibliography{final}

\end{document}